\documentclass[a4paper,11pt,reqno]{amsart}%

\usepackage{amsmath}%
\usepackage{amsfonts}%
\usepackage{amssymb}%
\usepackage{amsbsy}
\usepackage{amsthm}
\usepackage{cite}
\usepackage{graphicx}
\usepackage[noprefix]{nomencl}
\usepackage{setspace}
\usepackage{chngcntr}
\usepackage{thmtools}
\usepackage{dsfont}

\DeclareFontFamily{U}{mathx}{\hyphenchar\font45}
\DeclareFontShape{U}{mathx}{m}{n}{
      <5> <6> <7> <8> <9> <10>
      <10.95> <12> <14.4> <17.28> <20.74> <24.88>
      mathx10
      }{}
\DeclareSymbolFont{mathx}{U}{mathx}{m}{n}
\DeclareFontSubstitution{U}{mathx}{m}{n}
\DeclareMathAccent{\widecheck}{0}{mathx}{"71}
\DeclareMathAccent{\wideparen}{0}{mathx}{"75}

\makeatletter
\g@addto@macro{\thm@space@setup}{\thm@headfont{\bfseries}}
\makeatletter

\makeatletter
\def\th@plain{%
  \thm@notefont{}
  \itshape 
}
\def\th@definition{%
  \thm@notefont{}
  \normalfont 
}
\makeatother

\let\PROOF=\proof
\renewcommand\proof{\PROOF[\bfseries\proofname]}

\theoremstyle{plain}

\newtheorem{lemma}{Lemma}[section]
\newtheorem{corollary}[lemma]{Corollary}

\newtheorem{definition}[lemma]{Definition}

\newtheorem{proposition}[lemma]{Proposition}

\newtheorem{theorem}[lemma]{Theorem}
\numberwithin{equation}{section}

\makenomenclature
\makeindex

\renewcommand{\O}{{\rm O}}

\newcommand{\rmE}{{\rm E}}
\newcommand{\GL}{{\rm GL}}

\newcommand{\bfE}{{\bf E}}

\newcommand{\bfA}{{\bf A}}

\newcommand{\bfQ}{{\bf Q}}
\newcommand{\bfI}{{\bf I}}

\newcommand{\R}{{\mathbb R}}

\newcommand{\C}{{\mathbb C}}
\newcommand{\Z}{{\mathbb Z}}

\newcommand{\T}{{\mathbb T}}

\newcommand{\E}{{\rm E}}

\newcommand{\rmT}{{\rm T}}

\newcommand{\rmP}{{\rm P}}

\newcommand{\calB}{{\mathcal{B}}}
\newcommand{\calE}{{\mathcal{E}}}
\newcommand{\calL}{{\mathcal{L}}}

\newcommand{\calH}{{\mathcal{H}}}

\newcommand{\calR}{{\mathcal{R}}}

\newcommand{\calZ}{{\mathcal{Z}}}

\renewcommand{\mod}{{\rm mod}}
\newcommand{\res}{{\rm res}}

\newcommand{\supp}{{\rm supp}}

\renewcommand{\d}{\,{\rm d}}

\newcommand{\id}{{\rm id}}
\newcommand{\tr}{{\rm tr}}

\newcommand{\Aut}{{\rm Aut}}

\newcommand{\HS}{{\rm HS}}

\newcommand{\U}{{\rm U}}

\newcommand{\mult}{{\rm mult}}
\newcommand{\beq}{\begin{equation}}
\newcommand{\eeq}{\end{equation}}
\newcommand{\beqs}{\begin{eqnarray}}
\newcommand{\eeqs}{\end{eqnarray}}

\begin{document}
\numberwithin{equation}{section}
\counterwithout*{equation}{subsection}
\counterwithin{figure}{section}
\counterwithout*{figure}{subsection}

\title[The Zak transform on $G$-spaces]{The Zak transform on strongly proper $G$-spaces and its applications}
\author[D. J\"ustel]{Dominik J\"ustel}
\address{Technical University of Munich,\hspace*{\fill}\linebreak \indent Faculty of Mathematics,\hspace*{\fill}\linebreak \indent Boltzmannstr. 3,\hspace*{\fill}\linebreak \indent 85747 Garching b. M\"unchen}%
\email{juestel@ma.tum.de}
\thanks{DJ was partially supported by the TUM Graduate School. Part of the paper was developped during the author's stay at the HIM Trimester Program ``Mathematics of Signal Processing" in Bonn.}

\begin{abstract}
The Zak transform on $\R^d$ is an important tool in condensed matter physics, signal processing, time-frequency analysis, and harmonic analysis in general. This article introduces a generalization of the Zak transform to a class of locally compact $G$-spaces, where $G$ is either a locally compact abelian or a second countable unimodular type I group. This framework unifies previously proposed generalizations of the Zak transform. It is shown that the Zak transform has invariance properties analog to the classic case and is a Hilbert space isomorphism between the space of $L^2$-functions and a direct integral of Hilbert spaces that is explicitly determined via a Weil formula for $G$-spaces and a Poisson summation formula for compact subgroups. Some applications in physics are outlined.
\end{abstract}
\vspace{3mm}

\maketitle


\setcounter{tocdepth}{2}
\tableofcontents

\section{Introduction}
The Zak transform on $\R$ has already been known to Gelfand \cite{Gelfand1950} and Weil \cite{Weil1964}, and is known under many different names, e.g. the Weil-Brezin map in abstract harmonic analysis \cite{Brezin1970}, or the $kq$-representation \cite{Zak1967} or Bloch(-Floquet) transform  in physics.

It first generated considerable interest, when Zak \cite{Zak1967} rediscovered it and applied it in condensed matter physics as a refinement of the decomposition of electron states into so-called Bloch waves \cite{Bloch1929}. It also found applications in signal analysis (see e.g. \cite{Janssen1988}), in particular as a tool in Gabor analysis, where it is for example used to prove the Balian-Low theorem (see \cite{Groechenig2001} for this approach, and \cite{Balian1981,Low1985} for the original work).

The Zak transform has been generalized to locally compact abelian (lca) groups by Kaniuth and Kutyniok \cite{Kaniuth1998} and, by an ingenious construction, to certain non-abelian locally compact Hausdorff (lcH) groups in \cite{Kutyniok2002}, where clearly a generalization of Gabor analysis was the motivation. More recently, a Zak transform on certain semidirect product groups \cite{Arefijamaal2013} and for actions of lca groups \cite{Barbieri2015,Saliani2014} were considered.

Here, we propose a generalization of the Zak transform in the spirit of the work by Zak \cite{Zak1967}, but without requiring the `symmetry' group to be a subset of the space on which it acts. Instead, the group acts on a topological space in a suitable way. This approach allows, for example, to study rotational and helical symmetries in three-dimensional space in addition to the classic case of translational symmetries, where the group $\R^3$ acts on itself by translations. Along the orbits of the action, the Zak transform decomposes a function via Fourier analysis on non-abelian groups. In this sense, the Zak transform might be called a orbit-frequency decomposition.

The construction includes most previous approaches, but does for example not agree with Kutyniok's construction in the case of a subgroup of a non-abelian group acting by translation, illustrating the different needs for different applications. Our construction will naturally lead to a decomposition of function spaces and operators (e.g. eigenspaces of the electronic Schr\"odinger equation and the Hamiltonian) related to a class of non-crystalline molecular structures, called objective structures, that were studied by James in \cite{James2006} (see Section \ref{subsec:Bloch_Floquet}). The Zak transform originally emerged in the author's work on the design of electromagnetic radiation for the analysis of molecular structures \cite{Friesecke2016,Juestel2016}. This application is outlined in Section \ref{subsec:radiation_design}

The topological spaces considered in this article are strongly proper $G$-spaces. These are lcH spaces $X$ on which an lcH group $G$ acts continuously and properly. In addition, the orbit space of a strongly proper action is paracompact -- a property that is essential to relate integration on $X$ to the group action. This framework is introduced in Section \ref{subsec:proper_G_spaces}.

When there is a measure on $X$ that is quasi-invariant w.r.t. the strongly proper action of $G$ (see Definition \ref{def:invariant_measures}), then integration on $X$ can be decomposed into integration over the group and over the orbit space. This generalization of the classic Weil formula (Theorem \ref{thm:Weil_formula}) for the integration on homogeneous spaces is proved in Section \ref{subsec:Weil_formula}, and nicely complements results by Bourbaki \cite{Bourbaki2004}. It is also closely related to integration over fundamental domains of the action (see Section \ref{subsec:fundamental_domain}).

The Zak transform on $X$ w.r.t. a strongly proper group action is then defined for the action of a locally compact abelian (lca) group (Section \ref{sec:Zak_transform_abelian}) and for the action of a second countable unimodular type I group (Section \ref{sec:Zak_transform_nonabelian}). The fact that the introduced transforms are isometries essentially reduces to an application of the Weil formula and the Plancherel theorem. In the abelian case, the image of the Zak transform can explicitly be determined by combining the Poisson summation formula with the Plancherel theorem (Theorem \ref{thm:Zak_abelian}). In the case of non-abelian groups, a generalized Poisson summation formula is formulated that allows to show the analog result (Theorem \ref{thm:Zak_nonabelian}).

These results explicitly perform the decomposition of the action on the space of $L^2$-functions into irreducible representations. For abelian actions, functions can be decomposed into invariant functions that are suitably modulated along orbits of the action. For non-abelian actions, the Zak transform can be viewed as a family of tensor fields on a fundamental domain of the action. Here the underlying vector space of the tensor field is not the tangential space at the respective point, but the representation space of an element of the dual $\widehat{G}$. A function can then be decomposed into invariant tensor fields that are modulated along the orbits by irreducible representations of the group. We call these equivariant fields {\it Bloch tensor fields}. Moreover, invariant differential operators decompose accordingly. In other words, the Zak transform embodies the representation theory of the action and allows to explicitly determine its decomposition.

In the spirit of Zak's approach \cite{Zak1967}, the Zak transform is written as a decomposition of a function into equivariant measures supported on orbits. This viewpoint reveals the details of the decomposition into Bloch tensor fields. In Zak's words, the Bloch fields are ``expressible in infinitely localized Wannier functions" \cite{Zak1967}.

Finally, in Section \ref{sec:Zak_applications} some applications are discussed, namely the Bloch decomposition of the electronic eigenspaces of objective structures, and the appearance of the Zak transform in radiation design (see \cite{Friesecke2016,Juestel2016}).

Readers that are not interested in the details of the Weil formula can essentially skip Section \ref{sec:Weil_formula} and directly proceed to Section \ref{sec:Zak_transform_abelian} after having a look at Definition \ref{def:proper_action}, where a strongly proper action is defined, at Theorem \ref{thm:Weil_formula}, where the Weil formula for strongly proper group actions is found, and at eqs. \eqref{measure_fundamental_domain} and \eqref{integration_fundamental_domain}, where the measure on the fundamental domain and its relation to the Weil formula are explained.

\section{The Weil formula for strongly proper $G$-spaces} \label{sec:Weil_formula}

\subsection{Strongly proper $G$-spaces} \label{subsec:proper_G_spaces}
The goal of this section is to introduce the setting of this paper, namely the framework of strongly proper $G$-spaces. These are topological spaces on which a topological group $G$ acts in a particularly nice way. The definition of properness goes back to Palais \cite{Palais1960} and was motivated by the attempt to generalize results for compact group actions to a more general setting. For our purpose, the importance of the stronger condition of strong properness of an action lies in the fact that the orbit spaces are paracompact.

Different authors proposed different inequivalent notions of properness that are tailored to specific needs. In \cite{Biller2004}, Biller found a way to reformulate and generalize these definitions in the framework of so-called Cartan actions (a notion that also goes back to Palais) via properties of the orbit space of the action.

Let $G$ be a topological group acting continuously on a topological space $X$ via the action $\rho:G\times X\to X$. We write $\rho_g(x) := \rho(g,x)$ for $g \in G$ and $x \in X$, and $\rho_G(x) := \{\rho_g(x) \,|\, g \in G\}$ for $x \in X$. Moreover, let $\rho\backslash X := \{\rho_G(x) \,|\, x \in X\}$ be the orbit space, equipped with the quotient topology, i.e. the final topology w.r.t. the quotient map $\pi_G:X\to\rho\backslash X$, $x\mapsto \rho_G(x)$. Furthermore, the set $G_{A,B} := \{g \in G \,|\, \rho_g(A)\cap B\not= \emptyset\}$ is called the transporter of the two subsets $A,B\subset X$, and $G_x := G_{\{x\},\{x\}}$ denotes the stabilizer of $x \in X$.

Recall that a topological space is called Hausdorff, if points can be separated by neighborhoods, called regular, if points and closed sets can be separated by neighborhoods, and called paracompact, if every open cover has an open locally finite refinement.

\begin{definition}[Cartan action, (Strongly/Palais-)proper action] \label{def:proper_action}
Let $G$ be a topological group acting continuously on a topological space $X$ via the action $\rho:G\times X\to X$.
\begin{enumerate}
\item[(i)] $\rho$ is called Cartan action, if the stabilizers $G_x$ are compact for all $x \in X$, and for every $x \in X$ and every neighborhood $U$ of $G_x$ there is a neighborhood $V$ of $x$, s.t. $G_{V,V}\subseteq U$.
\item[(ii)] $\rho$ is called proper action, if it is a Cartan action and $\rho\backslash X$ is Hausdorff.
\item[(iii)] $\rho$ is called Palais-proper action, if it is a Cartan action and $\rho\backslash X$ is regular.
\item[(iv)] $\rho$ is called strongly proper action, if it is a Cartan action and $\rho\backslash X$ is paracompact.
\end{enumerate}
\end{definition}

Condition (i) is a continuity condition on the map $x\mapsto G_x$. As seen in \cite{Biller2004}, definitions (i) and (iii) generalize the respective definitions of Palais \cite{Palais1960}, (ii) is equivalent to Bourbaki's definition \cite[III, \S 4.1]{Bourbaki1989}, and (iv) generalizes the definition of Baum, Connes, and Higson \cite{Baum1994}.

When dealing with integration on proper $G$-spaces, one considers locally compact Hausdorff (lcH) groups acting on lcH spaces. In this setting, the orbit space is again lcH (see\cite[VII,\S 2]{Bourbaki2004}), so in particular Hausdorff and regular, s.t. (i)-(iii) in Definition \ref{def:proper_action} are equivalent. Moreover, in \cite[III, \S 4.4]{Bourbaki1989} it is shown that in this case, properness of the action is equivalent to compactness of all transporters of compact sets.

Strong properness, i.e. paracompactness of the orbit space, is important to relate integration on $X$ to integration on $\rho\backslash X$, as will be seen in Section \ref{sec:Weil_formula}. In \cite{Chabert2001}, Chabert, Echterhoff, and Meyer showed that all four definitions (i)-(iv) in Definition \ref{def:proper_action} are equivalent when $G$ and $X$ are lcH and second countable. More generally, Biller \cite{Biller2004} showed that the action of a locally compact Lindel\"of group on a paracompact lcH space is strongly proper, if and only if it is proper\footnote{In fact, Abels \cite{Abels1974} conjectured that every Palais-proper action of a connected lcH group on a paracompact space is strongly proper.} (A topological space is called Lindel\"of, if every open cover has a countable subcover.).

We give an illustrating example for the different notions of properness. Consider the action $\rho$ of $\Z$ on $\R^2$ that is given by
$$\rho_n(x) = (2^n x_1,2^{-n} x_2),\ \ \ n \in \Z,\ x = (x_1,x_2) \in \R^2.$$
Note, that interpreting $\R^2$ as the phase space $\R\times\widehat{\R}$, this action is a symplectic transformation w.r.t. the standard symplectic form $\omega((x_1,y_1),(x_2,y_2)) = x_1y_2 - x_2y_1$.

The action $\rho$ is not a Cartan action, because the stabilizer of the origin $(0,0)$ is the whole group, which is not compact. Restricting the action to $\R^2\setminus\{0\}$ yields a Cartan action, because the stablizers are then all trivial, and the continuity condition follows from continuity of the action. However, this action is not proper, because the quotient space is not Hausdorff. This is easiest seen from the fact that the graph of the action is not closed (see \cite[Prop. 8, I, 8.3]{Bourbaki1989}). Alternatively, consider two sets $A$ and $B$ that are neighborhoods of points on the $x$-axis and the $y$-axis, respectively. Then on can show that the transporter $\Z_{A,B}$ is not compact. Now, when further restricting the action to the set $\R^2\setminus\{(x,y) \,|\, x = 0\ \mbox{or}\ y = 0\}$, then the action is proper, as the quotient space is the topological sum of four half-lines. So, this action is also Palais-proper and strongly proper, what can also be seen with the result of Chabert et al. \cite{Chabert2001}, since $\Z$ and $\R^2$ are second countable. This example shows that properness is a quite strong condition, as it already fails for this simple discrete abelian automorphic action. However, it is needed for the integration theory on $G$-spaces, as seen in the following section.

\subsection{The Weil formula}
\label{subsec:Weil_formula}
A Weil formula is a tool to relate integration on a $G$-space $X$ to integration on the group $G$ that acts on $X$. This is achieved by constructing a suitable measure on the space of orbits.

In the case of $X=G$ being a lcH group with (left) Haar measure $\mu_G$, and a closed subgroup $H$ with Haar measure $\mu_H$ acting on $G$ via right translation $R:H\times G\to G$, $R_h(g) = gh$, $h \in H$, $g \in G$, a classic result is, that there always is a (left $G$-quasi-invariant) Radon measure $\mu_{G/H}$ on the quotient space $G/H$, s.t. the Weil formula
\begin{equation} \label{Weil_formula_G/H}
\int_G f\cdot q\d\mu_G = \int_{G/H}A_Rf\d\mu_{G/H},\ \ \ f \in C_c(G),
\end{equation}
holds, where the orbital mean operator $A_R:C_c(G)\to C_c(G/H)$ is given by
\begin{equation} \label{orbital_mean_operator}
(A_Rf)(g) := \int_Hf(gh)\d\mu_H(h),\ \ \ f \in C_c(G),\ g \in G,\ h \in H
\end{equation}
(see \cite{Bourbaki2004,Reiter2000,Folland1995}). Here, $q$ is a continuous and strictly positive function that satisfies a certain functional equation (and is sometimes called a $\rho$-function, see \cite{Folland1995}). We will prove a generalization of this Weil formula to strongly proper $G$-spaces under some additional assumptions. It nicely complements results on the Weil formula due to Bourbaki (see \cite[VII,2]{Bourbaki2004}). Even though the result might be known to experts, I didn't find it in the literature. The proofs are more or less standard and can be found in the appendix.

A continuous action $\rho$ of a lcH group $G$ on a lcH space $X$ defines a linear action of $G$ on the vector space of Borel-measurable functions $f$ on $X$ via
\begin{equation} \label{action_on_functions}
(\rho_g f)(x) := f(\rho_g^{-1}(x)),\ \ \ g \in G,\ x \in X.
\end{equation}
As the action is continuous, equation \eqref{action_on_functions} also defines a linear action on the space $C_c(X)$ of compactly supported continuous functions.

We start with the following definitions.

\begin{definition}[Invariant measures] \label{def:invariant_measures}
Let $X$ be a lcH $G$-space on which $G$ acts continuously via $\rho:G\times X\to X$. A Borel measure $\mu$ on $X$ is called
\begin{itemize}
\item[(i)] $\rho$-quasi-invariant, if there are continuous functions $\lambda_g:X\to \C$, $g \in G$, s.t.
$$\int_X \rho_gf\d\mu = \int_X f\cdot\lambda_g\d\mu\ \ \ \mbox{for all}\ f \in C_c(X),$$
\item[(ii)] relatively $\rho$-invariant, if it is $\rho$-quasi-invariant, and the functions $\lambda_g$ are constant,
\item[(iii)] $\rho$-invariant, if it is relatively $\rho$-invariant, and $\lambda_g \equiv 1$ for all $g \in G$.
\end{itemize}
\end{definition}
As examples for these properties, consider a Haar measure $\mu_G$ on a locally compact group $G$. Then $\mu_G$ is invariant w.r.t. the action of $G$ by left translation, relatively invariant w.r.t. right translation and quasi-invariant w.r.t. the action of $\Z_2$ by inversion:
\begin{align*}
\d\mu_G(hg) = \d\mu_G(g),\ \ \ \d\mu_G(gh) = \Delta_G(h)\d\mu_G(g),\ \ \ \d\mu_G(g^{-1}) = \Delta_G(g)^{-1}\d\mu_G(g),
\end{align*}
for $g,h \in G$, where $\Delta_G$ denotes the modular function of $G$, that is defined by the second equality. A Haar measure is also relatively invariant with respect to automorphisms. Let $\varphi \in \Aut(G)$, then there is a positive number $\mod(\varphi)$, called the modulus of $\varphi$, s.t. $d\mu_G(\varphi(g)) = \mod(\varphi)\d\mu_G(g)$ for $g \in G$.

Note that in \cite{Folland1995}, the notion of quasi-invariance used here is called strong quasi-invariance, while quasi-invariance means mutual absolute continuity of the measures $B\mapsto \mu_X(\rho_g(B))$, $g \in G$. The approach of this article avoids additional technical issues.

A desirable property of the orbital mean operator \eqref{orbital_mean_operator} is to map into the space of continuous and compactly supported functions on the space of orbits, which allows to invoke the Riesz representation theorem for the construction of a suitable measure for a Weil formula. A sufficient condition for this property is properness of the action.

\begin{lemma}[Orbital mean operator] \label{lem:orbital_mean_operator}
Let $\rho:G\times X\to X$ be a proper continuous action of an lcH group $G$ on an lcH space $X$. The orbital mean operator $A_\rho$, given by
$$A_\rho f(\pi_\rho(x)) := \int_G\rho_gf(x)\d\mu_G(g),\ \ \ f \in C_c(X),\ x \in X.$$
is well-defined and maps a function $f \in C_c(X)$ to a function $A_\rho f \in C_c(\rho\backslash X)$.
\end{lemma}
\begin{proof}
The proof is standard and can be found in the appendix.
\end{proof}

We are now in place to state the main theorem of this section. It complements results of Bourbaki for proper actions and generalizes classic results on quotients $G/H$.

\begin{theorem}[Weil formula for strongly proper $G$-spaces] \label{thm:Weil_formula}
Let $\rho:G\times X\to X$ be a strongly proper action of an lcH group $G$ on an lcH space $X$. When there is a $\rho$-quasi-invariant Radon measure $\mu_X$ on $X$ with functions $\lambda_g$, s.t. $\d\mu_X(\rho_g(x)) = \lambda_g\d\mu_X(x)$ for $g \in G$, $x \in X$, then there is a strictly positive and continuous function $q$ on $X$ that satisfies
\begin{equation} \label{functional_eq_q}
\rho_gq(x) = \frac{\Delta_G(g)}{\lambda_{g^{-1}}(x)}q(x),\ \ \ g \in G,\ x \in X,
\end{equation}
and a unique Radon measure $\mu_{\rho\backslash X}$ on $\rho\backslash X$, s.t.
\begin{equation} \label{Weil_formula}
\int_X f\cdot q\d\mu_X = \int_{\rho\backslash X}A_\rho f\d\mu_{\rho\backslash X},\ \ \ f \in C_c(X).
\end{equation}
In particular, when $\mu_X$ is relatively $\rho$-invariant with $\lambda_g\equiv \Delta_G(g^{-1})$, then \eqref{Weil_formula} is satisfied with $q\equiv 1$.
\end{theorem}
\begin{proof}
The proof can be found in the appendix.
\end{proof}
Note that the measure $q\d\mu_X$ is realtively $\rho$-invariant with $\lambda_g\equiv\Delta_G(g)$, since by the functional equation \eqref{functional_eq_q} for $q$,
$$\rho_{g}q(x)\d\mu(\rho_{g}(x)) = \Delta_G(g)q(x)\d\mu(x),\ \ \ g \in G,\ x \in X.$$
This shows that the measure $q\d\mu_X$ is the unique measure $\mu_{\rho\backslash X}^\#$ in \cite[Prop. 4, VII, 2.2]{Bourbaki2004}, that satisfies the Weil formula.

For future reference, we introduce a name for the setting in Theorem \ref{thm:Weil_formula}.

\begin{definition}[Weil $G$-space] \label{def:Weil_G_space}
Let $G$ be an lcH group that acts strongly proper on an lcH space $X$ via the action $\rho:G\times X\to X$, and $\mu_X$ a $\rho$-quasi-invariant measure on $X$.

The measure space $(X,\mu_{\rho\backslash X}^\#)$ is called a Weil $G$-space, if there is a measure $\mu_{\rho\backslash X}$ on $\rho\backslash X$ and a function $q$ on $X$ that satisfies \eqref{functional_eq_q}, s.t. the Weil formula \eqref{Weil_formula} holds, and $d\mu_{\rho\backslash X}^\#(x) = q(x)\d\mu_X(x)$, $x \in X$.
\end{definition}

A technical tool that is needed to prove Theorem \ref{thm:Weil_formula}, is a so-called Bruhat function.

\begin{definition}[Bruhat function] \label{def:Bruhat_function}
Let $\rho:G\times X\to X$ be a proper continuous action of a lcH group $G$ on a lcH space $X$.

A continuous function $\beta$ on $X$ is called a Bruhat function for $\rho$, if
\begin{enumerate}
\item[(i)] $A_\rho\beta \equiv 1$,
\item[(ii)] for every compact set $C \subseteq X$, the restriction of $\beta$ to $\rho_G(C)$ is non-negative and compactly supported.
\end{enumerate}
\end{definition}

Thus, a Bruhat function is a preimage of the constant function $1$ under the orbital mean operator that has nice compactness properties. It will be used to define the function $q$ in \eqref{Weil_formula}. The existence of a Bruhat function is guaranteed by the following Lemma that is in this generality due to Bourbaki and goes back to Bruhat (see \cite{Bruhat1956}).

\begin{lemma}[Existence of a Bruhat function] \label{lem:Bruhat_function}
Let $\rho:G\times X\to X$ be a strongly proper action of an lcH group $G$ on an lcH space $X$. Then there is a Bruhat function for $\rho$.
\end{lemma}
\begin{proof}
A proof can be found in \cite[VII, 2.4]{Bourbaki2004}.
\end{proof}

Note that the proofs of this lemma for quotient spaces $G/H$ in \cite{Reiter2000} and \cite{Folland1995} use paracompactness implicitly.

A Bruhat function nicely relates invariant functions on $X$ to functions on $\rho\backslash X$. Given a function $f \in C_c(X)$, we associate to it the $\rho$-invariant function $f_\rho$ on $X$ that is given by $f_\rho(x) := A_\rho f(\pi_\rho(x))$. Then
\begin{equation} \label{orbital_mean_of_invariant}
A_\rho(\beta\cdot f_\rho)(\pi_\rho(x)) = \int_G \rho_g\beta(x)\rho_gf_\rho(x)\d\mu_G(g) = A_\rho f(\pi_\rho(x))\int_G\rho_g\beta(x)\d\mu_G(g) = A_\rho f(\pi_\rho(x)).
\end{equation}
This shows that in our setting one can associate a function on $\rho\backslash X$ to an invariant function on $X$ without reference to a specific fundamental domain. Moreover, this implies surjectivity of $A_\rho:C_c(X)\to C_c(\rho\backslash X)$, since for $f \in C_c(\rho\backslash X)$, the function $x\mapsto \beta(x)f(\pi_\rho(x))$ is a preimage that lies in $C_c(X)$.

A simple example for the Weil formula \eqref{Weil_formula} is the case of a closed subgroup $H$ of a lcH group $G$ acting by left translation $L_h(g) := hg$. The space of orbits is the quotient space $H\backslash G$. The Haar measure $\mu_G$ is $L$-invariant, s.t. a function $q$ is needed that satisfies $q(h^{-1}g) = \Delta_G(h)q(g)$. This is true for $q = \Delta_G^{-1}$, showing that there is a unique Radon measure $\mu_{H\backslash G}$ on $H\backslash G$, s.t.
\begin{equation} \label{Weil_formula_left}
\int_Gf\d\mu_G^{-1} = \int_Gf\cdot\Delta_G^{-1}\d\mu_G = \int_{H\backslash G}A_Lf\d\mu_{H\backslash G},\ \ \ f \in C_c(G).
\end{equation}

Theorem \ref{thm:Weil_formula} can be generalized to $L^1(X,\mu_X)$ via classic measure theoretic arguments (as in \cite{Reiter2000} for quotients $G/H$). In particular, one can show that every $L^1$-function on $X$ is in $L^1(G)$ along almost all orbits. This generalization is sometimes called the extended Weil formula. Moreover, with a modified orbital mean operator
$$A_{\rho,q}f(\pi_\rho(x)) := \int_G\frac{\rho_gf}{\rho_gq}\d\mu_G(g),\ \ \ f \in L^1(G),\ x \in X,$$
the following version of \eqref{Weil_formula} holds:
\begin{equation} \label{Mackey_Bruhat_formula}
\int_X f\d\mu_X = \int_{\rho\backslash X} A_{\rho,q}f\d\mu_{\rho\backslash X},\ \ \ f \in L^1(X,\mu_X).
\end{equation}
This is sometimes called the (extended) Mackey-Bruhat formula.

The extended Weil formula can be used to study integrability of functions along orbits. In particular, we get the following results for $L^p$-functions that will be crucial for the definition of the Zak transform on $L^2(X)$ in Sections \ref{sec:Zak_transform_abelian} and \ref{sec:Zak_transform_nonabelian}.

\begin{corollary}[Orbital mean on $L^p(X)$] \label{cor:orbital_mean_L2}
Let $G$ be an lcH group that acts strongly proper on a Weil $G$-space $(X,\mu_{\rho\backslash X}^\#)$ via $\rho:G\times X\to X$. Let $f \in L^p(X,\mu_{\rho\backslash X}^\#)$ for $1\leq p \leq \infty$. Then
\begin{enumerate}
\item[(i)] The function $(\pi_\rho(x),g)\mapsto \rho_gf(x)$ is in $L^p((\rho\backslash X)\times G)$.
\item[(ii)] For almost every $\pi_\rho(x) \in \rho\backslash X$, the function $f_x:G\to\C$, $g\mapsto \rho_gf(x)$, is in $L^p(G)$.
\end{enumerate}
\end{corollary}
\begin{proof}
(i) and (ii): Using the extended Weil formula on the function $|f|^p \in L^1(X,\mu_{\rho\backslash X^\#})$, we get
\begin{align*}
\infty > \|f\|_p^p &= \int_X|f|^p\d\mu_{\rho\backslash X}^\# = \int_{\rho\backslash X}\int_G |\rho_gf(x)|^p\d\mu_G(g)\d\mu_{\rho\backslash X}(\pi_\rho(x)),
\end{align*}
showing (i) and that $\pi_\rho(x)\mapsto \int_G|f_x|^p\d\mu_G$ is finite for almost every $\pi_\rho(x) \in \rho\backslash X$, i.e. that (ii) is true.
\end{proof}

\subsection{Integration on fundamental domains}
\label{subsec:fundamental_domain}
The Weil formula can be used to construct measures on measurable fundamental domains of the action. A fundamental domain is a set $F\subseteq X$, s.t. the restriction of $\pi_\rho$ to $F$ is a bijection. In other words, it is a set of representatives of the equivalence relation $x\sim_\rho y :\Leftrightarrow \pi_\rho(x) = \pi_\rho(y)$.

Unimodularity of $G$ turns out to be necessary to relate integration over $F$ to integration over $\rho\backslash X$ in a nice way. In particular, it makes sure that the integral of a $\rho$-invariant function over different fundamental domains agrees. Thus, {\it for the rest of this section assume that $G$ is unimodular}.

Now, let $(X,\mu_{\rho\backslash X}^\#)$ be a Weil $G$-space. For $f \in C_c(X)$, we use the Weil formula twice, first for the action $\rho$ on $X$, and then inside the orbital mean integral for the action of $G_x$ on $G$ by left translation $L_{g_x}(g) := g_xg$, $g_x \in G_x$, $g \in G$. Since $G$ is unimodular, the latter Weil formula is particularly simple (see eq. \eqref{Weil_formula_left}).
\begin{align*}
\int_X f\d\mu_{\rho\backslash X}^\# &= \int_{\rho\backslash X}\int_G \rho_g f\d\mu_G(g)\d\mu_{\rho\backslash X} \\
&= \int_{\rho\backslash X}\int_{G_x\backslash G}\int_{G_x}\rho_{g_x^{-1}g}f(x)\d\mu_{G_x}(g_x)\d\mu_{G_x\backslash G}(\pi_L(g))\d\mu_{\rho\backslash X}(\pi_\rho(x)) \\
&= \int_{\rho\backslash X}\mu_{G_x}(G_x)\int_{G_x\backslash G}\rho_g f(x)\d\mu_{G_x\backslash G}(\pi_L(g))\d\mu_{\rho\backslash X}(\pi_\rho(x)).
\end{align*}
The choice of the representatives of the orbits does not influence the volume of the stabilizer, since stabilizers of different points on the same orbit are conjugated via $G_{\rho_g(x)} = g G_x g^{-1}$, $g \in G$. Unimodularity of $G$ yields that the map $x\mapsto \mu_{G_x}(G_x)$ is $\rho$-invariant. For convenience, we normalize $\mu_{G_x\backslash G}$ and $\mu_{G_x}$ s.t. $\mu_{G_x}(G_x) = 1$ for all $x \in X$.

Now, given a measurable fundamental domain $F\subseteq X$ of $\rho$, we define a measure $\mu_F$ on $F$ as the pushforward of $\mu_{\rho\backslash X}$ w.r.t. the map $\varphi:\rho\backslash X\to F$, $\pi_\rho(x_0)\mapsto x_0$.
\begin{equation} \label{measure_fundamental_domain}
\mu_F := \varphi_*\mu_{\rho\backslash X}.
\end{equation}
Here, the pushforward measure of a measure $\mu$ by a measurable function $\psi$ is defined by $\psi_*\mu(B) := \mu(\psi^{-1}(B))$ for measurable sets $B$. The fact that integrals of $\rho$-invariant functions over different fundamental domains agree is clear from the definition.

In summary, we have shown that for every measurable fundamental domain $F$,
\begin{equation} \label{integration_fundamental_domain}
\int_X f\d\mu_{\rho\backslash X}^\# = \int_F\int_{G_{x_0}\backslash G}\rho_g f(x_0)\d\mu_{G_{x_0}\backslash G}(\pi_L(g))\d\mu_F(x_0),\ \ \ f \in L^1(X).
\end{equation}
This result agrees with the construction for discrete groups in \cite[VII, 2.10]{Bourbaki2004}.

Note that the function $f_x:G\to\C$, $f_x(g) := \rho_g f(x)$, is constant on cosets $G_x g$, so it is in fact well-defined as a function on $G_x\backslash G$. The orbit $\rho_G(x)$ is homeomorphic to the homogeneous space $G_x\backslash G$, so we can construct measures $\mu_{\rho_G(x)} := \gamma_*\mu_{G_x\backslash G}$, $\gamma:G_x g\mapsto \rho_g(x)$, on the orbits of $\rho$. This analysis shows that when almost all stabilizers are trivial and $G$ is $\sigma$-compact, then the Weil formula actually reduces to Fubini's theorem.

Viewing the orbits as homogeneous spaces also allows a formulation of the Weil formula in its form \eqref{integration_fundamental_domain} as a transform (which should probably be called the Weil transform) on $L^2(X)$. A function $f \in L^2(X)$ is mapped to the function $(x_0,\pi_L(g))\mapsto \rho_gf(x_0)$ for $x_0 \in F$, and $\pi_L(g) \in G_{x_0}\backslash G$. Applying the extended Weil formula to the function $|f|^2 \in L^1(X)$ shows that this transform is a unitary map from $L^2(X)$ onto the direct integral space $\int_F^\oplus L^2(G_{x_0}\backslash G)\d\mu_F(x_0)$.

The question if there is a `nice' measurable fundamental domain is a difficult question for general actions. As shown in \cite[VII, Ex. 12]{Bourbaki2004}, there is a fundamental domain whose characteristic function is continuous almost everywhere (or equivalently, whose boundary is negligable) for a proper continuous action of a countable discrete group. For lattice subgroups of lcH groups, one can even find relatively compact fundamental domains as shown in \cite{Kaniuth1998,Kutyniok2002}.

The choice of a fundamental domain and the definition \eqref{measure_fundamental_domain} are important to define the Zak transform (see Definitions \ref{def:Zak_abelian} and \ref{def:Zak_nonabelian}), since the definition will depend on this choice.

\section{The Zak transform for abelian actions} \label{sec:Zak_transform_abelian}

In this section, the Zak transform for abelian actions will be defined, its main properties will be derived, and different interpretations of its meaning will be discussed. The mathematical tools used to analyze the transform are the Weil formula that was derived in Section \ref{sec:Weil_formula}, and Fourier analysis on lca groups. As will turn out, the representation theoretic decomposition of $\rho$ is fully described by the Zak transform.

\subsection{The abelian Zak transform on $L^1(X)$} 
\label{subsec:Zak_abelian_L1}

Before defining the Zak transform for abelian actions, we recall the main objects and results from Fourier analysis on lca groups.
Let $G$ be an lca group. The dual group $\widehat{G}$ of $G$ is the set of characters of $G$, which are the continuous homorphisms from $G$ to $\T := \{z \in \C \,|\, |z| = 1\}$.
\begin{equation} \label{dual_group}
\widehat{G} := \{\chi:G\to\T \,|\, \chi\ \mbox{continuous},\ \chi(gh) = \chi(g)\chi(h),\ g,h \in G\}.
\end{equation}
With the point-wise product and the compact open topology, $\widehat{G}$ is an lca group. Given a Haar measure $\mu_G$ of $G$, the Fourier transform on an lca group is then defined as
\begin{equation} \label{Fourier_lca}
\widehat{f}(\chi) := \int_G f(g)\overline{\chi(g)}\d\mu_G(g),\ \ \ f \in L^1(G),\ \chi \in \widehat{G}.
\end{equation}
The Fourier transform can be inverted as follows. For almost every $g \in G$,
\begin{equation} \label{Fourier_inversion}
f(g) = \int_{\widehat{G}}\widehat{f}(\chi)\chi(g)\d\mu_{\widehat{G}}(\chi),\ \ \ f \in L^1(G),\ \widehat{f} \in L^1(\widehat{G}).
\end{equation}

Another construction that will be used is the so-called reciprocal group $H_G^\bot$ of a closed subgroup $H$ of $G$.
\begin{equation} \label{reciprocal_group}
H_G^\bot := \{\chi \in \widehat{G} \,|\, \chi(h) = 1\ \mbox{for all}\ h \in H\},
\end{equation}
which is also called the orthogonal subgroup in \cite{Reiter2000} or the annihilator of $H$ in \cite{Hewitt1979} and generalizes the reciprocal lattice of a lattice subgroup of $\R^d$. The reciprocal group is a central object in the Poisson summation formula for lca groups, which reads
\begin{equation} \label{Poisson_summation}
\int_H f\d\mu_H = \int_{H_G^\bot}\widehat{f}\d\mu_{H_G^\bot},\ \ \ f|_H \in L^1(H),\ \widehat{f}|_{H_G^\bot} \in L^1(H_G^\bot).
\end{equation}

We define the Zak transform of a function as the mean of its modulation by a character along the orbits.

\begin{definition}[Zak transform for abelian actions] \label{def:Zak_abelian}
Let $G$ be an lca group, $(X,\mu_{\rho\backslash X}^\#)$ a Weil $G$-space, and $f \in L^1(X,\mu_{\rho\backslash X}^\#)$. Furthermore, let $F$ be a measurable fundamental domain of the action $\rho$. The Zak transform $\calZ_\rho f$ of $f$ is defined as
\begin{equation} \label{Zak_abelian}
\calZ_\rho f(x_0,\chi) := \int_G \rho_gf(x_0)\overline{\chi(g)}\d\mu_G(g),\ \ \ x_0 \in F,\ \chi \in \widehat{G}.
\end{equation}
We also define the extended Zak transform $\calZ_\rho^X f$ on $X\times \widehat{G}$ by formula \eqref{Zak_abelian} with $x_0 \in X$.
\end{definition}

Note, that the transform is well-defined, since an $L^1$-function is in $L^1(G)$ along $\mu_{\rho\backslash G}$-almost all orbits. This fact justified the extension of the Weil formula to $L^1(X)$.

{\bf Example:} Consider the discrete abelian group $G = \bfA\Z^d$, $\bfA \in \GL(d,\R)$, acting on $X = \R^d$ by translation $\rho_x(y) = x+y$, $x,y \in \R^d$. The Lebesgue measure $\calL$ on $\R^d$ is translation-invariant, s.t. it can be chosen as the measure $\mu_{\rho\backslash X}^\# := \calL$. The measures on $G$ and $\rho\backslash X$ that satisfy the Weil formula are then the counting measure and the measure induced from the restriction of $\calL$ to a fundamental domain $F$, respectively. Denoting the characters of $G$ by $\chi^{(k)}(x) := e^{ik\cdot x}$, $x \in \R^3$, $k \in \R^3/G'\cong \widehat{G}$, where $G' := 2\pi\bfA^{-T}\Z^d$ denotes the reciprocal lattice (which parametrizes the reciprocal group $G_{\R^3}^\bot$), the resulting Zak transform $\calZ_\rho$ is the classic Zak transform
\begin{equation} \label{classic_Zak}
\calZ_\rho f(x_0,\chi^{(k)}) := \sum_{v \in \bfA\Z^d} f(x_0-v)e^{-ik\cdot v},\ \ \ x_0 \in F,\ k \in \R^3/G',\ f \in L^1(\R^3).
\end{equation}

To simplify notation we chose the notation $\calZ_\rho$ for the Zak transform, even though one has to keep in mind that the definition depends on the choice of fundamental domain $F$, so a notation like $\calZ_\rho^F$ would be more precise. The reason for restricting the first argument to a fundamental domain is the following equivariance property of the extended Zak transform.

\begin{lemma}[Invariance of the extended Zak transform] \label{lem:Zak_invariance}
With the assumptions of Definition \ref{def:Zak_abelian}, the extended Zak transform satisfies
\begin{equation} \label{Zak_invariance}
\calZ_\rho^Xf(\rho_g^{-1}(x_0),\chi) = \chi(g)\calZ_\rho f(x_0,\chi),\ \ \ g \in G,\ x_0 \in F,\ \chi \in \widehat{G}.
\end{equation}
In particular, the Zak transform vanishes on the set $\{(x_0,\chi) \,|\, \chi\not\in (G_{x_0})_G^\bot\}$.
\end{lemma}
\begin{proof} Equation \eqref{Zak_invariance} is an immediate consequence of the definition \eqref{Zak_abelian}, using left-invariance of $\mu_G$.

Now, when $x_0 \in F$, and $\chi \not\in (G_{x_0})_G^\bot$, then there is a $g \in G_{x_0}$, s.t. $\chi(g)\not= 1$. In particular, using eq. \eqref{Zak_invariance},
$$\calZ_\rho f(x_0,\chi) = \calZ_\rho f(\rho_g^{-1}(x_0),\chi) = \chi(g)\calZ_\rho f(x_0,\chi),$$
implying that $\calZ_\rho f(x_0,\chi) = 0$.
\end{proof}

The invariance property \eqref{Zak_invariance} shows that the extended Zak transform $\calZ_\rho^X f$ is $\rho$-equivariant in the first argument w.r.t. the second argument. So, the operator $\calZ_\rho^X$ maps to the space of functions with this property. This viewpoint captures the fact that the extended Zak transform is uniquely determined by the restriction to any fundamental domain abstractly. However, in applications it is convenient to work with a special choice of fundamental domain.

The fact that the Zak transform vanishes for $(x_0,\chi)$ with $\chi\not\in(G_{x_0})_G^\bot$ is another manifestation of the fact that the functions $f_{x_0}:g \mapsto \rho_gf(x_0)$ are actually functions on $G_{x_0}\backslash G$. Since it can be shown that $\widehat{G_{x_0}\backslash G}\cong (G_{x_0})_G^\bot$, this vanishing property is a consequence of the Poisson summation formula on $G$ w.r.t. $G_{x_0}$. In principle, one could define the Zak transform of a function as a function on the set $\{(x_0,\chi) \,|\, x_0 \in F,\ \chi \in (G_{x_0})_G^\bot\}$. For the moment, the idea is to keep the space on which $\calZ_\rho f$ lives simple and capture the details in its properties. These considerations will be made precise in Theorem  \ref{thm:Zak_abelian}.

Note that in the case that $X=G$ is an abelian lcH group with a lattice $H$ acting by right translation, Definition \ref{def:Zak_abelian} agrees up to the conjugation of the character with the one given in \cite{Kaniuth1998}. To see this, one needs to identify a fundamental domain of $H_G^\bot$ with the group $\widehat{H}$ via the isometry $\widehat{G}/H_G^\bot \cong \widehat{H}$ (see e.g. \cite[(4.39)]{Folland1995}). Also the Zak transform on semi-direct product groups with the action of a lattice subgroup of the abelian factor in \cite{Arefijamaal2013} is a special case of our definition. In \cite{Barbieri2015} and \cite{Saliani2014}, the authors propose a more general representation theoretic approach. They consider a unitary representation $\sigma$ of a discrete lca group $G$ on $L^2(X)$ and the transform $\sum_{g \in G}\sigma_g f(x)\overline{\chi}(g)$, $x \in X$, $\chi \in \widehat{G}$ for $f \in L^2(X)$. This approach is natural in frame theory, but too general for applications that use the additional structure resulting from the action on the underlying space considered here.

The definition of the Zak transform can be interpreted in different ways. First, it may be viewed as the Fourier transform along the orbits of the action. Given a function $f \in L^1(X)$, the Zak transform is given by
\begin{equation} \label{Zak_Fourier}
\calZ_\rho f(x_0,\chi) = \widehat{f_{x_0}}(\chi),\ \ \ x_0 \in F,\ \chi \in \widehat{G},
\end{equation}
where $f_{x_0}(g) := \rho_g f(x_0)$, $x_0 \in F$, $g \in G$. From this viewpoint, the Zak transform could be termed an orbit-frequency decomposition.

The complementary point of view results from fixing a character $\chi \in \widehat{G}$. Then the operator $P_\rho^\chi: f\mapsto \calZ_\rho^Xf(\cdot,\chi)$ is a so-called symmetry projection, and
\begin{equation} \label{Zak_Wigner}
\calZ_\rho^X f(x,\chi) = P_\rho^\chi f(x),\ \ \ x \in X,\ \chi \in \widehat{G}.
\end{equation}
As seen in Lemma \ref{lem:Zak_invariance}, $P_\rho^\chi$ maps $f$ into the space of $\chi$-equivariant functions on $X$. This approach goes back to Wigner \cite{Wigner1931}, who used these operators in quantum theory. The same viewpoint is taken by Bourbaki \cite[VII, \S 2]{Bourbaki1989}, where instead of characters, representations on the real multiplicative group are considered.

Both viewpoints can be summarized using the Heisenberg group $H(G) := G\ltimes_\tau (\widehat{G}\times\T)$ of $G$, where $\tau[g](\chi,z) := (\chi,\chi(g)z)$, and the group product is given by
\begin{equation} \label{product_Heisenberg}
(g,(\chi,z))\cdot(g',(\chi',z')) := (gg',\tau[g](\chi,z)\cdot(\chi',z')) = (gg',(\chi\chi',\chi(g)zz'))
\end{equation}
for $g,g' \in G$, $\chi,\chi' \in \widehat{G}$, and $z,z'\in \T$. The Heisenberg group is the group theoretic version of a generalized time-frequency plane or phase space. More precisely, phase space is the quotient $H(G)/\T$, where $\T$ is identified with the center $Z(H(G)) = \{(e,(1,z)) \,|\, z \in \T\}$ of $H(G)$. This quotient is in fact isomorphic to the direct product of group and dual group:
\begin{equation} \label{phase_space}
H(G)/Z(H(G)) \cong G\times\widehat{G}.
\end{equation}
Now the projective representation $\xi$ of $G\times \widehat{G}$ on $L^2(X)$, that is given by
\begin{equation} \label{proj_Schroedinger_rep}
\xi_{(g,\chi)}f(x) := \rho_gf(x)\overline{\chi(g)},\ \ \ g \in G,\ \chi \in \widehat{G}.
\end{equation}
naturally defines a representation of $H(G)$, namely (up to conjugation of the character) the classic Schr\"odinger representation $\xi^{H(G)}$ of the Heisenberg group that is the lift of $\xi$ to the central extension $H(G)$ of $G\times\widehat{G}$:
\begin{equation} \label{Schroedinger_rep}
\xi_{(g,(\overline{\chi},z))}^{H(G)}f(x) := \rho_gf(x)\chi(g)z,\ \ \ g \in G,\ \chi \in \widehat{G},\ z \in \T.
\end{equation}
The Zak transform can be interpreted as a partial orbital mean operator w.r.t. $\xi$:
\begin{equation} \label{Zak_Heisenberg}
\calZ_\rho f(x_0,\chi) = \int_G\xi_{(g,\chi)}f(x)\d\mu_G(g),\ \ \ x_0 \in F,\ \chi \in \widehat{G}.
\end{equation}
This viewpoint is the reason for the importance of the classic Zak transform for time-frequency analysis and in particular for Gabor analysis.

Next, we use the Fourier interpretation \ref{Zak_Fourier} to formulate an inversion formula. We simply use the Fourier inversion formula \eqref{Fourier_inversion} for the second argument of the Zak transform to get
\begin{equation} \label{Zak_inversion_1}
f(x) = \int_{\widehat{G}}\calZ_\rho f(x_0,\chi)\chi(g)\d\mu_{\widehat{G}}(\chi),\ \ \ \mbox{where}\ \rho_g(x) = x_0 \in F,
\end{equation}
almost everywhere for functions $f \in L^1(G)$ with $\calZ_\rho f(x_0,\cdot) \in L^1(\widehat{G})$ for almost all $x_0 \in F$.

Note that the right hand side is well-defined, even though the equation $x = \rho_g^{-1}(x_0)$ has multiple solutions, namely all $(x_0,g'g)$ with $g' \in G_{x_0}$. As seen in Lemma \ref{Zak_invariance}, the Zak transform of a function vanishes for all $\chi \in \widehat{G}$ that are not constant on stabilizers. Consequently, choosing a measure $\mu_{(G_{x_0})_G^\bot}$ on $(G_{x_0})_G^\bot\cong\widehat{G_{x_0}\backslash G}$, s.t. the Poisson summation formula for the subgroup $G_{x_0}$ of $G$ holds with normalized measures on $G_{x_0}$, $x_0 \in F$, one gets the alternative version
\begin{equation} \label{Zak_inversion_2}
f(x) = \int_{(G_{x_0})_G^\bot}\calZ_\rho f(x_0,\chi)\chi(g)\d\mu_{(G_{x_0})_G^\bot}(\chi),\ \ \ \mbox{where}\ \rho_g(x) = x_0 \in F,
\end{equation}
for $f$ and $x$ as above.

To rigorously define the inverse Zak transform, we formulate the condition $\rho_g(x) = x_0 \in F$ in terms of measures. Consider the point measures $\delta_x$ on $F$ and $\delta_x^{(x_0)}$ on $G_{x_0}\backslash G$ that are given by
\begin{equation}
\delta_x(B_F) = \begin{cases} 1, & \mbox{there is a}\ g \in G,\ \mbox{s.t.}\ \rho_g(x) \in B_F, \\ 0, & \mbox{else},\end{cases}
\end{equation}
for Borel-sets $B_F\subseteq F$, and
\begin{equation}
\delta_x^{(x_0)}(B_{G_{x_0}\backslash G}) = \begin{cases} 1, & \mbox{there is a}\ \pi_L(g) \in B_{G_{x_0}\backslash G},\ \mbox{s.t.}\ \rho_g(x) = x_0, \\ 0, & \mbox{else},\end{cases}
\end{equation} 
for Borel-sets $B_{G_{x_0}\backslash G}\subseteq G_{x_0}\backslash G$. With these measures we summarize the above considerations in the following result.

\begin{proposition}[$L^1$-Zak inversion theorem] \label{prop:Zak_inversion}
Let $G$ be an lca group that acts strongly proper on a Weil space $(X,\mu_{\rho\backslash X}^\#)$, and $F$ a measurable fundamental domain of the action $\rho$. The inverse Zak transform $\calZ_\rho^{-1}$ is defined as
\begin{equation} \label{Zak_inversion}
\calZ_\rho^{-1} h(x) := \int_F\int_{(G_{x_0})_G^\bot}h(x_0,\chi)\int_{G_{x_0}\backslash G}\chi(g)\d\delta_x^{(x_0)}(\pi_L(g))\d\mu_{(G_{x_0})_G^\bot}(\chi)\d\delta_x(x_0)
\end{equation}
for $h(x_0,\cdot) \in L^1((G_{x_0})_G^\bot)$ for almost all $x_0 \in F$. It satisfies $f = \calZ_\rho^{-1}\calZ_\rho f$ point-wise almost everywhere for $f \in L^1(G)$ with $\calZ_\rho f(x_0,\cdot)\in L^1((G_{x_0})_G^\bot)$ for almost all $x_0 \in F$.
\end{proposition}
\begin{proof}
The proof is a direct consequence of the above discussion, the equations \eqref{Zak_Fourier} and \eqref{Zak_inversion_2}, and the $L^1$-Fourier inversion formula \eqref{Fourier_inversion}.
\end{proof}
The inversion formula \eqref{Zak_inversion} can be read as a decomposition of a function into the measures $\int_{G_{x_0}\backslash G}\chi(g)\d\delta_x^{(x_0)}$ on $F\times G_{x_0}\backslash G$. This is a generalization of Zak's viewpoint as given in \cite[(5), (12)]{Zak1967} and will be discussed in more detail below Theorem \ref{thm:Zak_abelian}.

{\bf Example:} Let's reconsider the classic case that the lattice $G = \bfA\Z^d$, $\bfA \in \GL(d,\R)$, acts on $X = \R^d$ by translation. Since the stabilizers of this action are trivial, we have $G_{x_0}\backslash G = G$ and $(G_{x_0})^\bot_G = \widehat{G}$ for all $x_0 \in F$, where $F$ is a relatively compact fundamental domain of the action. The measure $\delta^{(x_0)}_x$ on $G$ is simply the Kronecker delta $\delta_{x_0-x,v}$. Consequently, the Zak reconstruction formula reads
\begin{equation} \label{classic_Zak_inversion}
\calZ_\rho^{-1}h(x) = \int_F\int_{\R^3/G'}h(x_0,\chi^{(k)})\sum_{v \in G}e^{ik\cdot v}\delta_{x_0-x,v}\d\mu_{\R^3/G'}(k)\d\delta_x(x_0),\
\end{equation}
where again $G' = 2\pi\bfA^{-T}\Z^d$ denotes the reciprocal lattice of $G$, s.t. $\R^3/G'\cong \widehat{G}$ parametrizes the dual group, and $\mu_{\R^3/G'}$ is the measure induced by this parametrization. This agrees with Zak's inversion formula \cite[(5),(12)]{Zak1967} up to conjugation of the character. The constant in Zak's formula is contained in the normalization of the measure $\mu_{\R^3/G'}$.

\subsection{The abelian Zak transform on $L^2(X)$} 
\label{subsec:Zak_abelian_L2}

Using Corollary \ref{cor:orbital_mean_L2}, we can extend the Zak transform to $L^2(G)$ and show that it is an isometry. For this purpose, we need the Plancherel theorem for lca groups which says that the Fourier transform is a Hilbert space isomorphism from $L^2(G)$ to $L^2(\widehat{G})$. This can be formulated in a representation theoretic way as follows. The Fourier transform is a Hilbert space isomorphism from $L^2(G)$ to $\int_{\widehat{G}}^\oplus\C\d\mu_{\widehat{G}} = L^2(\widehat{G})$, where the integral is a direct integral of the measurable field of one-dimensional Hilbert spaces $\C$ on which the characters $\chi \in \widehat{G}$ act via the irreducible unitary modulation representations $M_\chi:G\to \U(\C)$, $M_\chi(g)z := \chi(g)z$, where $\U(\C)$ denotes the group of continuous unitary operators on $\C$. The Fourier transform intertwines the left regular representation $L$ (left translation, $L_gf(h) := f(g^{-1}h)$) of $G$ on $L^2(G)$ with the direct integral representation $\int_{\widehat{G}}^\oplus \overline{M_\chi}\d\mu_{\widehat{G}}(\chi)$ on $L^2(\widehat{G})$. This is a precise formulation of the statement that the Fourier transform diagonalizes the translation operator.

The corresponding theorem for the Zak transform is the following.

\begin{theorem}[Abelian Zak transform on $L^2(X)$] \label{thm:Zak_abelian}
Let $G$ be an abelian lcH group that acts strongly proper on a Weil $G$-space $(X,\mu_{\rho\backslash X}^\#)$ via $\rho:G\times X\to X$ and $F$ a measurable fundamental domain of the action $\rho$. The Zak transform can be extended to a Hilbert space isomorphism
\begin{equation} \label{Zak_isometry}
\calZ_\rho:L^2(X)\to\int_F^\oplus\int_{(G_{x_0})_G^\bot}^\oplus\C\d\mu_{(G_{x_0})_G^\bot}\d\mu_F(x_0) = \int_F^\oplus L^2((G_{x_0})_G^\bot)\d\mu_F(x_0),
\end{equation}
where $\mu_F$ is given by eq. \eqref{measure_fundamental_domain}, and $\mu_{(G_{x_0})_G^\bot}$ is normalized, s.t. the Poisson summation formula \eqref{Poisson_summation} for $G_{x_0}$ holds with $\mu_{G_{x_0}}(G_{x_0}) = 1$ for $x_0 \in F$.

The Zak transform intertwines the action $\rho$ of $G$ on $L^2(X)$ with the left modulation representation of $G$ on $\int_F^\oplus L^2((G_{x_0})_G^\bot)\d\mu_F(x_0)$:
\begin{equation} \label{Zak_intertwining}
\calZ_\rho[\rho_g f](x_0,\chi) = \chi(g)\calZ_\rho f(x_0,\chi),\ \ \ g \in G,\ x_0 \in F,\ \chi \in (G_{x_0})_G^\bot,
\end{equation}
for $f \in L^2(G)$.

The inverse Zak transform $\calZ_\rho^{-1}$ is given by the extension of eq. \eqref{Zak_inversion} to $\int_F^\oplus L^2((G_{x_0})_G^\bot)\d\mu_F(x_0)$.
\end{theorem}

\begin{proof}{\bf (of Theorem \ref{thm:Zak_abelian})} By Corollary \ref{cor:orbital_mean_L2}, the functions $f_x$, $g\mapsto \rho_gf(x)$, are in $L^2(G)$ for almost all $\pi_\rho(x) \in \rho\backslash X$, so the Zak transform is well-defined almost everywhere by eq. \eqref{Zak_Fourier}.

We show that the Zak transform is an isometry by combining the Weil formula with Plancherel's theorem. For $f \in L^2(X,\mu_{\rho\backslash X}^\#)$, we use eq. \eqref{Zak_Fourier} to get
\begin{align*}
\|f\|_2^2 &\stackrel{\rm Weil}{=} \int_{\rho\backslash X}\int_G |f_x|^2\d\mu_G\d\mu_{\rho\backslash X}(\pi_\rho(x)) \stackrel{\rm Plancherel}{=} \int_{\rho\backslash X}\int_{\widehat{G}}|\widehat{f_x}|^2\d\mu_{\widehat{G}}\d\mu_{\rho\backslash X}(\pi_\rho(x)) \\
&\stackrel{\eqref{Zak_Fourier}}{=} \int_{\rho\backslash X}\int_{\widehat{G}}|\calZ_\rho f(x,\chi)|^2\d\mu_{\widehat{G}}(\chi)\d\mu_{\rho\backslash X}(\pi_\rho(x)).
\end{align*}
Now, by Lemma \ref{lem:Zak_invariance}, the function $x\mapsto \int_{\widehat{G}}|\calZ_\rho^X f(x,\chi)|\d\mu_{\widehat{G}}(\chi)$ is $\rho$-invariant, so with eq. \eqref{integration_fundamental_domain} for the fundamental domain $F$ (note that $G$ is abelian, so unimodular), we get
\begin{align*}
\|f\|_2^2 = \int_{F}\int_{\widehat{G}}|\calZ_\rho^X f(x,\chi)|^2\d\mu_{\widehat{G}}(\chi)\d\mu_F(x) = \int_{F}\int_{\widehat{G}}|\calZ_\rho f(x_0,\chi)|^2\d\mu_{\widehat{G}}(\chi)\d\mu_F(x_0).
\end{align*}
Furthermore, using the fact that $\calZ_\rho f$ vanishes for $\chi \not\in (G_{x_0})_G^\bot$ together with the Weil formula for $G_{x_0}\backslash G$ and the normalization of $\mu_{(G_{x_0})_G^\bot}$, we get
\begin{align*}
\|f\|_2^2 = \int_F\int_{(G_{x_0})_G^\bot}|\calZ_\rho(x_0,\chi)|^2\d\mu_{(G_{x_0})_G^\bot}(\chi)\d\mu_F(x_0) = \|\calZ_\rho f\|_{L^2_\rho(F\times\widehat{G})}^2.
\end{align*}
This shows that the Zak transform is a linear isometry into $L^2_\rho(F\times\widehat{G})$.

Next, we show that the extension is a Hilbert space isomorphism by explicitly writing down the preimage of a function $h \in L^2_\rho(F\times\widehat{G})$. Consider the equivariant extension $h^X$ of $h$ to $X\times\widehat{G}$ that is given by $h^X(x,\chi) = \chi(g)h(x_0,\chi)$ for $x \in X$ and $\chi \in (G_{x_0})_G^\bot$, where $x_0$ and $g$ are given by $\rho_g(x) = x_0$. Note again, that, since $\chi \in (G_{x_0})_G^\bot$, it is constant on stabilizer cosets, so the extension is well-defined.

Then $|h^X(x,\chi)|^2$ is $\rho$-invariant and the above calculation can be carried out backwards. The inverse Plancherel-Fourier transform of $\chi\mapsto h(x_0,\chi)$ yields a function $\widecheck{h}\in L^2(F\times G)$. It needs to be checked, that this function is of the form $(x_0,g)\mapsto f(\rho_g^{-1}(x_0))$ for some $f \in L^2(X)$ and almost all $x_0 \in F$. But this is true, since $h(x_0,\chi)$ is only supported on $(G_{x_0})_G^\bot$, so $\widecheck{h}(x_0,\cdot)$ is $G_{x_0}$-invariant by the Plancherel-theorem for the group $G_{x_0}\backslash G$, using the isomorphism $\widehat{G_{x_0}\backslash G}\cong (G_{x_0})_G^\bot$.

Finally, an application of the Weil formula yields that $\calZ_\rho^{-1}h \in L^2(X)$. That this transform inverts the Zak transform is immediate from eq. \eqref{Zak_Fourier} and Plancherel's theorem.

The intertwining property \eqref{Zak_intertwining} is shown by direct calculation, using the definition \eqref{Zak_abelian} and the translation invariance of the Haar measure.
\end{proof}

Using the Hausdorff-Young inequality for lca groups $\|\widehat{f}\|_q\leq\|f\|_p$ for $1 < p < 2$ and $q = p/(p-1)$ and $f \in L^1\cap L^p(G)$, instead of the Parseval identity, one obtains an $L^p$-Zak transform $\calZ_p :L^p(X,\mu_{\rho\backslash X}^\#)\to L^q(F\times\widehat{G},\mu_F\otimes\mu_{\widehat{G}})$, using the $L^p$-version of Corollary \ref{cor:orbital_mean_L2}.

We next consider Zak's approach to the Zak transform on $\R^d$ in \cite{Zak1967}, where he considers the transform as a decomposition of a function into simultaneous eigenfunctions of the group of translation operators associated to a lattice and the group of modulations associated to the reciprocal lattice. In fact, the objects he considers are not functions, but measures. More precisely, the classic Zak transform decomposes a function into Dirac combs on orbits of the lattice group that are modulated by characters belonging to a fundamental domain of the reciprocal lattice (the so-called first Brillouin zone). The same construction in the general abelian case works as follows.

Let $G$ be an lca group acting strongly proper on the Weil $G$-space $(X,\mu_{\rho\backslash X}^\#)$ via the action $\rho$, and let $F$ be a measurable fundamental domain of this action. Now, consider the locally finite complex Borel measures $\delta_G^{(x_0,\chi)}$ on $F\times\widehat{G}$, $x_0 \in F$, $\chi \in\widehat{G}$, that are defined by
\begin{equation} \label{Zak_measure}
\delta_G^{(x_0,\chi)}(\varphi) := \int_G\rho_g\varphi(x_0)\chi(g)\d\mu_G(g) = \calZ_\rho\varphi(x_0,\overline{\chi}),\ \ \ \varphi \in C_c(X).
\end{equation}
We call these measures Zak measures. The values of the Zak transform of a function $f \in L^2(X)$ can be considered as the coefficients of the expansion into these measures as seen by the following calculation. For $\varphi \in C_c(X)$,
\begin{align*}
\int_{F\times\widehat{G}}&\calZ_\rho f(x_0,\chi)\delta_G^{(x_0,\chi)}(\varphi)\d\mu_{F\times\widehat{G}}(x_0,\chi) = \langle\calZ_\rho f,\calZ_\rho\overline{\varphi}\rangle_{L^2(F\times\widehat{G})} = \langle f,\overline{\varphi}\rangle_{L^2(X)} = f(\varphi),
\end{align*}
where $f$ is interpreted as a functional on $C_c(X)$, namely $\varphi\mapsto f(\varphi) := \int_Xf\cdot\varphi\d\mu_X$. Here, we used the isometry property of the Zak transform, and $\overline{\calZ_\rho\overline{\varphi}(x_0,\chi)} = \calZ_\rho\varphi(x_0,\overline{\chi})$. This a weak inversion theorem in terms of measures. The strong version is Proposition \ref{prop:Zak_inversion}.

This shows that the Zak measures $\delta_G^{(x_0,\chi)}$, $x_0 \in F$, $\chi \in \widehat{G}$, are the analogs of Zak's `functions' $\psi_{kq}$ in \cite[(5)]{Zak1967}. Just as these, they are eigenfunctions of the action of $G$. For $\varphi \in C_c(X)$,
\begin{equation} \label{Zak_measure_eigenfunction_1}
\rho_g\delta_G^{(x_0,\chi)}(\varphi) = \delta_G^{(x_0,\chi)}(\rho_g^{-1}\varphi) = \calZ_\rho[\rho_g^{-1}\varphi](x_0,\overline{\chi}) = \chi(g)\calZ_\rho\varphi(x_0,\overline{\chi}) = \chi(g)\delta_G^{(x_0,\chi)}(\varphi).
\end{equation}
Before generalizing the results of this section to the non-abelian case, we note that an interesting question is, when the space $\int_F^\oplus L^2((G_{x_0})_G^\bot)\d\mu_F(x_0)$ actually is identical to $L^2(F\times\widehat{G})$. This is for example the case when for almost every $x_0 \in F$, the stabilizer $G_{x_0}$ is trivial. Then the vanishing property becomes irrelevant in the inversion formula, and one can invert any function in $L^2(F\times\widehat{G})$ as in the proof of Theorem \ref{thm:Zak_abelian}.

\section{The Zak transform for non-abelian actions} \label{sec:Zak_transform_nonabelian}

To define a Zak transform for non-abelian groups, we can use eq. \eqref{Zak_Fourier} in a setting where a Fourier transform is available. Moreover, to prove the analog of Theorem \ref{thm:Zak_abelian}, we need a Plancherel formula. A reasonable setting is that of second countable unimodular type I groups.

We quickly review the main results of harmonic analysis on second countable unimodular type I groups. For a detailed expositions of the theory, see for example \cite{Folland1995,Dixmier1982,Lipsman1974,Hewitt1970}.

Let $G$ be a locally compact group. The dual $\widehat{G}$ of $G$ is the set of equivalence classes of unitary irreducible representations of $G$.
\begin{equation} \label{dual}
\widehat{G} = \{[\sigma] \,|\, \sigma:G\to \U(\calH_\sigma)\ \mbox{irreducible}\},
\end{equation}
where $[\sigma]$ denotes the equivalence class w.r.t. the relation of unitary equivalence. Two unitary representations $\sigma$ and $\sigma'$ of $G$ are called unitary equivalent, if there is a bounded linear unitary operator $T:\calH_\sigma\to\calH_{\sigma'}$ that intertwines $\sigma$ and $\sigma'$, i.e. that satisfies $\sigma' = T\sigma T^{-1}$. To simplify notation, we will omit the brackets for the equivalence classes throughout the text. The dual is a topological space with the so-called Fell topology (see e.g. \cite[7.2]{Folland1995}).

The dual of a type I group has special properties that are usually formulated in representation theoretic terms. For second countable groups, there is a way to formulate this property equivalently in terms of the toplogy and the measurable structure of $\widehat{G}$. Namely, being type I is in this situation equivalent to $\widehat{G}$ being a $T_0$-space (Kolmogorov space), and to the Borel $\sigma$-algebra on $\widehat{G}$ being standard. The latter means that $\widehat{G}$ is measurably isomorphic to a Borel subset of a complete separable metric space \cite[Thm. 7.6]{Folland1995}. All abelian and compact lcH groups are type I, while a discrete group is type I, if and only it has a normal abelian subgroup of finite index \cite{Thoma1964}. The latter fact will be used in Section \ref{sec:Zak_applications} to apply the Zak transform to discrete subgroups of the Euclidean group $\rmE(3)$ of isometries of $\R^3$.

For a second countable unimodular type I group $G$, the Fourier transform on $L^1(G)$ is defined by
\begin{equation} \label{Fourier_nonabelian}
\widehat{f}(\sigma) := \int_Gf(g)\sigma(g)^*\d\mu_G(g),\ \ \ f \in L^1(G),\ \sigma \in \widehat{G},
\end{equation}
where $\mu_G$ is a left Haar measure on $G$. The single `Fourier coefficients' $\widehat{f}(\sigma)$, $\sigma \in \widehat{G}$, are unitary operators on a Hilbert space $\calH_\sigma$. The Fourier transform defines what is called a measurable field of operators $\{\widehat{f}(\sigma)\}_{\sigma\in \widehat{G}}$ on $\widehat{G}$ (see \cite[7.4]{Folland1995}). When $f \in L^1(G)\cap L^2(G)$, then these operators are Hilbert-Schmidt operators for $\mu_{\widehat{G}}$-almost all $\sigma \in \widehat{G}$, where $\mu_{\widehat{G}}$ is the measure on $\widehat{G}$, s.t. the inverse Fourier transform is given by
\begin{equation} \label{Fourier_inversion_nonabelian}
f(g) = \int_{\widehat{G}}\tr(\widehat{f}(\sigma)\sigma(g))\d\mu_{\widehat{G}}(\sigma),\ \ \ f \in L^1(G)\ \widehat{f} \in L^1(\widehat{G})
\end{equation}
for almost every $g \in G$ (see \cite{Lipsman1974} for more details).

One further ingredient that is needed to transfer the theory to the non-abelian case, is a Poisson summation formula for the quotients by the stabilizers. By properness of the action, it suffices to consider quotients by compact groups.

In the following section, a Poisson summation formula for a general compact subgroup $H$ of a second countable unimodular type I group will be constructed. The following set $H_G^\bot$, which we call reciprocal space of $H$ in $G$, will play a role as that of the reciprocal group in the abelian case.
\begin{equation} \label{reciprocal_space}
H_G^\bot := \{\sigma \in \widehat{G} \,|\, \mult(1,\res^G_H(\sigma)) \geq 1\},
\end{equation}
This can be intuitively understood by noting that at least the present invariant components will lead to some `resonance'.

\subsection{A Poisson summation formula for quotients by compact groups}
\label{subsec:Poisson_compact}

We prove an analog of the Poisson summation formula \eqref{Poisson_summation} for the case that $G$ is a second countable unimodular type I group, and $H$ is a compact closed subgroup. For this purpose, we need to develop some kind of harmonic analysis on the quotient space $G/H$. It will turn out that we will have to leave classic representation theory, since the objects that emerge are not unitary operators on some Hilbert space, but linear operators between different Hilbert spaces. Parts of the reasoning are well-known and related to harmonic analysis on Gelfand pairs (see e.g. \cite{Dieudonne1960}). Many of the notions have also been introduced by Farashahi in \cite{Farashahi2016} from a complementary viewpoint.

We write down the dual object $\widehat{G/H}$ of $G/H$, before developing the theory. The elements are simply the non-trivial orbital means of the unitary irreducible representations of $G$ w.r.t. the right action of $H$.
\begin{equation}
\widehat{G/H} := \{A_R\sigma \,|\, \sigma\in\widehat{G}\}\setminus\{0_{\calH_\sigma} \,|\, \sigma \in \widehat{G}\},
\end{equation}
where $0_{\calH_\sigma}$ denotes the trivial operator that maps all vectors in $\calH_\sigma$ to zero. Note that these operators are well-defined, because $H$ is compact. To compute these orbital means, we decompose the restricted representation $\res^G_H\sigma = \bigoplus_j\sigma_j$, $\sigma_j \in \widehat{H}$, and $\calH_\sigma = \bigoplus_j\calH_j$. Now define the part of $\sigma$ that reduces to the trivial representation $\sigma^{(1)} := \bigoplus_{j,\sigma_j\equiv 1}\id_{\calH_j}$ which is the identity on $\calH_\sigma^{(1)} := \bigoplus_{j,\sigma_j\equiv 1}\calH_j$. Moreover, let $\sigma^\bot$ be the restriction to the non-trivial components collected in $\calH_\sigma^\bot$, s.t.
$$\res^G_H\sigma = \sigma^{(1)}\oplus\sigma^\bot,\ \ \ \calH_\sigma = \calH_\sigma^{(1)}\oplus\calH_\sigma^\bot.$$
Then, for $\sigma \in \widehat{G}$,
\begin{align*}
A_R\sigma(gH) &= \int_H\sigma(gh)\d\mu_H(h) = \sigma(g)\int_H\res^G_H\sigma\d\mu_H = \sigma(g)\int_H\sigma^{(1)}\oplus\sigma^\bot\d\mu_H \\
&= \sigma(g)(A_R\sigma^{(1)}\oplus A_R\sigma^\bot)(eH) = \sigma(g)(\id_{\calH_\sigma^{(1)}}\oplus 0_{\calH_\sigma^\bot}),
\end{align*}
where the last equality results from the Schur orthogonality relations for the irreducible representations of $H$. To simplify notation, we set $P_\sigma^{(1)} := \id_{\calH_\sigma^{(1)}}\oplus 0_{\calH_\sigma^\bot}$, which is the orthogonal projection to $\calH_\sigma^{(1)}$. In particular, we find that $A_R\sigma = 0_{\calH}$ for $\sigma \not\in H_G^\bot$ (see \eqref{reciprocal_space}), because by definition the restrictions of these representations do not contain the trivial representation of $H$. In summary, we found that
\begin{equation} \label{dual_quotient}
\widehat{G/H} = \{\sigma_{G/H} := \sigma\cdot P_\sigma^{(1)} \,|\, \sigma \in H_G^\bot\}.
\end{equation}
In particular, the objects $\sigma_{G/H}$ can be interpreted as linear operators from $\calH_\sigma^{(1)}$ to $\calH_\sigma$, suggesting a more general approach than representation theory, where the structure of a space is represented by operators between Hilbert spaces and not by unitary operators on one Hilbert space.

As Hilbert-Schmidt operators are often identified with tensors in $\calH_\sigma\otimes\calH_{\overline{\sigma}}$ (as antilinear maps from $\calH_\sigma^* = \calH_{\overline{\sigma}}$ to $\calH_\sigma$), the operators $\sigma_{G/H} \in \widehat{G/H}$ can be considered as tensors in $\calH_\sigma^{(1)}\otimes\calH_{\overline{\sigma}}$. Consequently, $\sigma_{G/H}^* = (P_\sigma^{(1)})^*\sigma^*$, $\sigma \in \widehat{G}$, can be considered as a tensor in $\calH_\sigma\otimes\calH_{\overline{\sigma}}^{(1)}$.

Next, we define the Fourier transform of a function $f \in C_c(G/H)$. For this purpose, we let $f_G(g) := f(gH)$, $g \in G$, be the invariant extension of $f$ to $G$, and calculate its $G$-Fourier transform, using the classic Weil formula. For $\sigma \in \widehat{G}$,
\begin{align*}
\widehat{f_G}(\sigma) &= \int_Gf_G\cdot\sigma^*\d\mu_G \stackrel{\rm Weil}{=} \int_{G/H}\int_H f_G(gh)\sigma(gh)^*\d\mu_H(h)\d\mu_{G/H}(gH) \\
&= \int_{G/H}f(gH)A_R\sigma^*(gH)\d\mu_{G/H}(gH) = \int_{G/H}f\cdot\sigma_{G/H}^*\d\mu_{G/H}.
\end{align*}
This allows to define the Fourier transform on $L^1(G/H)$ as
\begin{equation} \label{Fourier_quotient}
\widehat{f}(\sigma_{G/H}) := \int_{G/H}f\sigma_{G/H}^*\d\mu_{G/H},\ \ \ f \in L^1(G/H),\ \sigma_{G/H} \in \widehat{G/H}.
\end{equation}
In particular, this implies that $\supp(\widehat{f_G})\subseteq H_G^\bot$ for $f \in L^1(G/H)$, s.t. we can define $\mu_{H_G^\bot} := \mu_{\widehat{G}}|_{H_G^\bot}$. Moreover, consider the map $\varphi:H_G^\bot\to\widehat{G/H}$, $\sigma\mapsto \sigma_{G/H}$, which is obviously surjective. It is also injective, since for $\sigma,\sigma' \in H_G^\bot$ of same dimension,
$$\varphi(\sigma) = \varphi(\sigma') \Leftrightarrow \sigma(g)\cdot P_\sigma^{(1)} = \sigma'(g)\cdot P_{\sigma'}^{(1)} \Leftrightarrow (\sigma'^{-1}\sigma)(g)\cdot P_\sigma^{(1)} = P_{\sigma'}^{(1)}.$$
Since the spaces $\calH_\sigma^{(1)}$ and $\calH_{\sigma'}^{(1)}$ are non-trivial, this implies that there is at least a one-dimensional invariant subspace of $\sigma'^{-1}\sigma$, which implies $\sigma' = \sigma$ by irreducibility. Thus $\varphi$ is bijective and we can endow $\widehat{G/H}$ with the final topology w.r.t. $\varphi$. Moreover, we define a measure on $\widehat{G/H}$ by $\mu_{\widehat{G/H}} := \varphi_*\mu_{H_G^\bot}$.

This allows us to prove the following inversion theorem for the Fourier transform on $L^1(G/H)$.

\begin{proposition}[$L^1$-Fourier inversion for quotients by compact groups] \label{prop:Fourier_inversion_quotient}
Let $G$ be a second countable unimodular type I group, $H$ a compact closed subgroup of $G$, and $f \in L^1(G/H)$ with $\widehat{f} \in L^1(G/H)$. Then for almost every $gH \in G/H$,
\begin{equation} \label{Fourier_inversion_quotient}
f(gH) = \int_{\widehat{G/H}}\tr\left(\widehat{f}(\sigma_{G/H})\sigma_{G/H}(gH)\right)\d\mu_{\widehat{G/H}}(\sigma_{G/H}).
\end{equation}
In particular, when $f$ is in addition continuous, then eq. \eqref{Fourier_inversion_quotient} is true for every $gH \in G/H$.
\end{proposition}
\begin{proof}
Again, we can transfer the calculation to $L^1(G)$ and use the Fourier inversion formula \eqref{Fourier_inversion} to get the result. Let $f_G(g) := f(gH)$, $g \in G$, then for almost every $gH \in G/H$,
\begin{align*}
f(gH) &= f_G(g) = \int_{\widehat{G}}\tr(\widehat{f_G}(\sigma)\sigma(g))\d\mu_{\widehat{G}}(\sigma) = \int_{H_G^\bot}\tr(\widehat{f}(\sigma_{G/H})\sigma(g))\d\mu_{H_G^\bot}(\sigma) \\
&= \int_{\widehat{G/H}}\tr(\widehat{f}(\sigma_{G/H})\sigma_{G/H}(gH))\d\mu_{\widehat{G/H}}(\sigma_{G/H}).
\end{align*}
When $f$ is continuous, then $f_G$ is continuous, so the point-wise inversion formula follows from the inversion theorem on $G$.
\end{proof}

Note, that the inversion formula \eqref{Fourier_inversion_quotient} implies a Gelfand-Raikov theorem for $\widehat{G/H}$. The elements separate any two points, as shown by the inversion formula for a function in $C_c(G/H)$ that does not agree on those two points (which exists by the Urysohn lemma, since $G/H$ is paracompact and thus normal).

Another consequence is the following analog of the Poisson summation formula.

\begin{proposition}[Poisson summation formula for compact quotients] \label{prop:Poisson_quotient}
Let $G$ be a second countable unimodular type I group, $H$ a compact closed subgroup of $G$, and $f \in C_c(G)$ with $f|_H \in L^1(H)$ and $\widehat{f} \in L^1(\widehat{G})$. Then
\begin{equation} \label{Poisson_quotient}
\int_H f(h)\d\mu_H(h) = \int_{H_G^\bot}\tr(P_\sigma^{(1)}\cdot\widehat{f}(\sigma))\d\mu_{H_G^\bot}(\sigma).
\end{equation}
\end{proposition}
\begin{proof}
We write the integral over $H$ as the evaluation of the orbital mean at the neutral element and apply the Fourier inversion formula \eqref{Fourier_inversion_quotient}.
\begin{align*}
\int_H f(h)\d\mu_H(h) &= A_R f(eH) = \int_{\widehat{G/H}}\tr(\widehat{A_R f}(\sigma_{G/H})\sigma_{G/H}(gH))\d\mu_{G/H}(gH) \\
&= \int_{H_G^\bot}\tr(\widehat{(A_R f)_G}(\sigma)\sigma(e))\d\mu_{H_G^\bot}(\sigma) = \int_{H_G^\bot}\tr(\widehat{(A_R f)_G}(\sigma))\d\mu_{H_G^\bot}(\sigma).
\end{align*}
We calculate the function in the trace.
\begin{align*}
\widehat{(A_R f)}_G(\sigma) &= \int_G(A_R f)_G(g)\sigma(g)^*\d\mu_G(g) = \int_G\int_H f(gh)\d\mu_H(h)\sigma(g)^*\d\mu_G(g) \\
&\stackrel{G\ {\rm unimod.}}{=} \int_G f(g)\int_H\sigma(gh^{-1})^*\d\mu_H(h)\d\mu_G(g) = P_\sigma^{(1)}\cdot\widehat{f}(\sigma),
\end{align*}
yielding eq. \eqref{Poisson_quotient}.
\end{proof}
The exact same arguments can be used to prove the analog results for the left quotient $H\backslash G$ with the reciprocal space
\begin{equation} \label{reciprocal_left}
\widehat{H\backslash G} = \{\sigma_{H\backslash G} := P_\sigma^{(1)}\cdot\sigma \,|\, \sigma\in H_G^\bot\}.
\end{equation}
The elements $\sigma_{H\backslash G}$ can also be considered as tensors, namely as elements of $\calH_\sigma\otimes\calH_{\overline{\sigma}}^{(1)}$. The resulting Poisson summation formula is identical to the version for right quotients. It will be used in Theorem \ref{thm:Zak_nonabelian} for the left quotient of $G$ by a stabilizer of the action.

Lipsman \cite{Lipsman1974} proved a Poisson summation formula for quotients by normal subgroups, s.t. the quotient group is unimodular. The special case that the group is in addition compact is contained in Proposition \ref{prop:Poisson_quotient}.

The above approach suggest an approach to harmonic analysis on quotient spaces in more generality. This, however, goes far beyond the purpose of this article and will be treated elsewhere.

Finally, we formulate the corollary that will allow to prove the analog of Theorem \ref{thm:Zak_abelian}.

\begin{corollary}[Plancherel theorem for $H$-invariant functions] \label{cor:Plancherel_quotient}
Let $G$ be a second countable unimodular type I group, and $H$ a compact subgroup. The Plancherel-transform is an isometry from the space of right $H$-invariant functions onto the space $\int_{\widehat{G/H}}^\oplus \calH_\sigma\otimes\calH_{\overline{\sigma}}^{(1)}\d\mu_{\widehat{G/H}}$, and an isometry from the space of left $H$-invariant functions onto the space $\int_{\widehat{H\backslash G}}^\oplus \calH_\sigma^{(1)}\otimes\calH_{\overline{\sigma}}\d\mu_{\widehat{H\backslash G}}$.
\end{corollary}
\begin{proof}
By compactness of $H$, the space of right $H$-invariant functions is a closed subspace of $L^2(G)$. In addition, it is isomorphic to $L^2(G/H)$ via $f\mapsto f\circ\pi_R$. The extension of the Fourier transform from $L^1(G/H)\cap L^2(G/H)$ to $L^2(G/H)$ is thus an isometry, because of Plancherel's theorem for $L^2(G)$. It maps into $\int_{\widehat{G/H}}^\oplus \calH_\sigma\otimes\calH_{\overline{\sigma}}^{(1)}\d\mu_{\widehat{G/H}}$, because of eq. \eqref{Fourier_quotient} and the identification of the operators with the respective tensors. The special form of the tensor spaces reflects the vanishing of the `columns' of the Fourier coefficients resulting from the Schur orthogonality relations.

Surjectivity is shown by embedding this image space into $\int_{\widehat{G}}^\oplus \calH_\sigma\otimes\calH_{\overline{\sigma}}\d\mu_{\widehat{G}}$ (by extension with zero columns), performing the inverse Plancherel transform and using eq. \eqref{Fourier_inversion_quotient}.

The analog argument yields the result for the space of left $H$-invariant functions.
\end{proof}

\subsection{The non-abelian Zak transform on $L^1(X)$}
\label{subsec:Zak_nonabelian_L1}

Now, we are in place to define the Zak transform on second countable unimodular type I groups and then prove the analog of Lemma \ref{lem:Zak_invariance}.

\begin{definition}[Zak transform for non-abelian actions] \label{def:Zak_nonabelian}
Let $G$ be a second countable unimodular type I group, $(X,\mu_{\rho\backslash X}^\#)$ a Weil $G$-space, and $f \in L^1(X,\mu_{\rho\backslash X}^\#)$. Furthermore, let $F$ be a $\mu_X$-measurable fundamental domain of the action $\rho$. The Zak transform $\calZ_\rho f$ of a function $f \in L^1(X,\mu_{\rho\backslash X}^\#)$ is defined as
\begin{equation} \label{Zak_nonabelian}
\calZ_\rho f(x_0,\sigma) := \int_G\rho_gf(x_0)\sigma(g)^*\d\mu_G(g),\ \ \ x_0 \in F,\ \sigma \in \widehat{G}.
\end{equation}
Again, the extended Zak transform $\calZ_\rho^X$ on $X\times\widehat{G}$ is given by eq. \eqref{Zak_nonabelian} with $x_0 \in X$.
\end{definition}

Note that Definition \ref{def:Zak_nonabelian} is a natural extension of Definition \ref{def:Zak_abelian} in the abelian case; the only change is that the characters $\chi$ are replaced by the, possibly higher-dimensional, irreducible unitary representations $\sigma$ of the group.

The Zak transform is a measurable field of operators on $F\times\widehat{G}$. In analogy to Lemma \ref{lem:Zak_invariance}, we can prove an equivariance property.

\begin{lemma}[Equivariance of the non-abelian extended Zak transform] \label{lem:Zak_invariance_nonabelian}
With the assumptions of Definition \ref{def:Zak_nonabelian}, the extended Zak transform satisfies
\begin{equation} \label{Zak_invariance_nonabelian}
\calZ_\rho^X f(\rho_g^{-1}(x_0),\sigma) = \calZ_\rho f(x_0,\sigma)\sigma(g),\ \ \ g \in G,\ x_0 \in F,\ \sigma \in \widehat{G}.
\end{equation}
In particular, $\calZ_\rho f(x_0,\sigma)$ vanishes on the set $\{(x_0,\sigma) \in F\times\widehat{G} \,|\, \sigma \not\in (G_{x_0})_G^\bot\}$.
\end{lemma}
\begin{proof} Equation \eqref{Zak_invariance_nonabelian} follows immediately from the definition \eqref{Zak_nonabelian}, using translation invariance of $\mu_G$.

Now, for $x_0 \in F$, let $\sigma$ be not contained in $(G_{x_0})_G^\bot$. Every $G_{x_0}$-irreducible component of $\res^G_H(\sigma)$ is then non-trivial. For $f \in L^1(G)$ and $G \in G_{x_0}$, using the equivariance property,
$$\calZ_\rho f(x_0,\sigma) = \calZ_\rho f(\rho_g^{-1}(x_0),\sigma) = \calZ_\rho f(x_0,\sigma)\sigma(g),$$
showing that $\calZ_\rho f(x_0,\sigma)$ vanishes on every $G_{x_0}$-invariant subspace by non-triviality of the respective component of $\res^G_H(\sigma)$. Since all components of the representation are non-trivial, $\calZ_\rho(x_0,\sigma) = 0$.
\end{proof}

As in the abelian case, the equivariance property \eqref{Zak_invariance_nonabelian} can be used to interpret the extended Zak transform as a map into the space of measurable fields of operators on $\widehat{G}$ that satisfy this property, without the necessity to specify a fundamental domain.

First, we consider the analogs of the interpretations \eqref{Zak_Fourier}, \eqref{Zak_Wigner}, and \eqref{Zak_Heisenberg}. The Fourier viewpoint is still valid. We have
\begin{equation} \label{Zak_Fourier_nonabelian}
\calZ_\rho f(x_0,\sigma) = \widehat{f_{x_0}}(\sigma),\ \ \ x_0 \in F,\ \sigma \in \widehat{G}.
\end{equation}
The same is true for the symmetry-projection viewpoint of Wigner. The operator $P_\rho^\sigma$ that is given by
\begin{equation} \label{Zak_Wigner_nonabelian}
\calZ^X_\rho f(x_0,\sigma) = P_\rho^\sigma f(x),\ \ \ x \in X,\ \sigma \in \widehat{G}
\end{equation}
maps $f$ into the space of $\sigma$-equivariant measurable fields of operators on $X$ (see Corollary \ref{cor:Bloch_Floquet}). However, we cannot formulate a version of the viewpoint \eqref{Zak_Heisenberg}, since the phase space $G\times\widehat{G}$ is not a group.

In analogy to Proposition \ref{prop:Zak_inversion}, one has the following inversion theorem.

\begin{proposition}[Non-abelian $L^1$-Zak inversion theorem] \label{prop:Zak_inversion_nonabelian}
Let $G$ be a second countable unimodular type I group that acts strongly proper on a Weil space $(X,\mu_{\rho\backslash X}^\#)$, and $F$ a measurable fundamental domain of the action $\rho$. The inverse Zak transform $\calZ_\rho^{-1}$ is defined as
\begin{equation} \label{Zak_inversion_nonabelian}
\calZ_\rho^{-1} h(x) := \int_F\int_{(G_{x_0})_G^\bot}\tr\left(h(x_0,\sigma)\int_{G_{x_0}\backslash G}\sigma(g)\d\delta_x^{(x_0)}(\pi_L(g))\right)\d\mu_{(G_{x_0})_G^\bot}(\sigma)\d\delta_x(x_0).
\end{equation}
for $h(x_0,\cdot) \in L^1(\widehat{G})$ for almost all $x_0 \in F$. It satisfies $f = \calZ_\rho^{-1}\calZ_\rho f$ point-wise almost everywhere for $f \in L^1(G)$ with $\calZ_\rho f(x_0,\cdot)\in L^1(\widehat{G})$ for almost all $x_0 \in F$.
\end{proposition}
\begin{proof}
The proof is a direct consequence of the above discussion, equation \eqref{Zak_Fourier_nonabelian}, the $L^1$-Fourier inversion formula \eqref{Fourier_inversion_nonabelian}, and the Poisson summation formula \eqref{Poisson_quotient}.
\end{proof}
Again, the inversion formula can be viewed as a decomposition of a function into measures as in \cite{Zak1967}. However, these measures are operator-valued in the non-abelian case. 

\subsection{The non-abelian Zak transform on $L^2(X)$}
\label{subsec:Zak_nonabelian_L2}

The $L^2$-theory of the Fourier transform is again captured in a Plancherel theorem.

For compact groups $G$, the Peter-Weyl theorem states that $L^2(G) = \bigoplus_{\sigma \in \widehat{G}}\calE_\sigma$, where $\calE_\sigma$ is the translation invariant subspace of $L^2(G)$ that is spanned by the matrix elements of $\sigma$. We can combine left and right translation in the two-sided regular representation $T$ of $G\times G$ on $L^2(G)$ ($T_{(h,h')}f(g) := f(h^{-1}gh')$), then the Fourier transform intertwines $T$ and the direct sum representation $\bigoplus_{\sigma \in \widehat{G}}(\sigma\otimes\overline{\sigma})^{\oplus d_\sigma}$, $(d_\sigma = \dim(\calH_\sigma))$, of $G\times G$ that acts on the space $\bigoplus_{\sigma \in \widehat{G}} (\calH_\sigma\otimes\calH_{\overline{\sigma}})^{\oplus d_\sigma}$ by left and right modulation ($M_\sigma:G\times G\to \U(\calH_\sigma\otimes\calH_{\overline{\sigma}})$, $M_\sigma(g,h)A := \sigma(g)A\sigma(h)^*$, $A \in \calH_\sigma\otimes\calH_{\overline{\sigma}}$), where we identify the Fourier coefficient $\widehat{f}(\sigma)$ with a tensor in $\calH_\sigma\otimes\calH_{\overline{\sigma}}$ by interpreting the linear operator as an antilinear map from $\calH_\sigma^* = \calH_{\overline{\sigma}}$ to $\calH_\sigma$. This can be done, since $\widehat{f}(\sigma)$ is a Hilbert-Schmidt operator, and is natural, when for example calculating matrix coefficients of the representation w.r.t. to a basis of $\calH_\sigma$ and the dual basis of $\calH_{\sigma}^* = \calH_{\overline{\sigma}}$.

The Plancherel theorem for second countable unimodular type I groups \cite[Thm. 7.44]{Folland1995} says essentially the same. Namely, that there is a unique measure $\mu_{\widehat{G}}$ on $\widehat{G}$, s.t. the Fourier transform \eqref{Fourier_nonabelian} is a Hilbert space isomorphism betwenn $L^2(G)$ and $\int_{\widehat{G}}^\oplus \calH_\sigma\otimes\calH_{\overline{\sigma}}\d\mu_{\widehat{G}}(\sigma)$, and that the Fourier inversion formula \eqref{Fourier_inversion_nonabelian} can be extended to $L^2(G)$. As for compact groups, the Fourier transform intertwines the two-sided regular representation with the two-sided modulation representation.

Before we can state the analog of Theorem \ref{thm:Zak_abelian}, we need to give a precise version of the refinement of the vanishing property of the Zak transform given after Lemma \ref{lem:Zak_invariance_nonabelian}.

Let $G$ be a second countable unimodular type I group and $\sigma \in \widehat{G}$ an irreducible representation of $G$. For $x_0 \in F$, the stabilizer $G_{x_0}$ is a compact subgroup of $G$ by properness of the action.

As we know that the functions $f_{x_0}:g\mapsto \rho_g f(x_0)$ are left $G_{x_0}$-invariant, we can apply Corollary \ref{cor:Plancherel_quotient} to show that the Plancherel transform maps $f_{x_0}$ to a tensor field in $\calH_\sigma^{(1,x_0)}\otimes\calH_{\overline{\sigma}}$ on $(G_{x_0})_G^\bot$, where the additional index $x_0$ reflects the dependence on the stabilizer $G_{x_0}$.

These considerations allow to formulate the following theorem.

\begin{theorem}[Non-abelian Zak transform on $L^2(X)$] \label{thm:Zak_nonabelian}
Let $G$ be a second countable unimodular type I group that acts strongly proper on a Weil $G$-space $(X,\mu_{\rho\backslash X}^\#)$ via $\rho:G\times X\to X$, and $F$ a measurable fundamental domain of the action $\rho$. The Zak transform can be extended to a Hilbert space isomorphism
\begin{equation} \label{Zak_isometry_nonabelian}
\calZ_\rho:L^2(X,\mu_{\rho\backslash X}^\#)\to \int_F^\oplus\int_{(G_{x_0})_G^\bot}^\oplus\calH_\sigma^{(1,x_0)}\otimes\calH_{\overline{\sigma}}\d\mu_{(G_{x_0})_G^\bot}(\sigma)\d\mu_F(x_0) =: L^{2,{\rm op}}_\rho(F\times\widehat{G}),
\end{equation}
where $\mu_F$ is given by eq. \eqref{measure_fundamental_domain}.

The Zak transform intertwines the action $\rho$ of $G$ on $L^2(X)$ with the modulation representation on $\calZ_\rho(L^2(X))$.
\begin{equation} \label{Zak_intertwining_nonabelian}
\calZ_\rho[\rho_g f](x_0,\sigma) = \sigma(g)\calZ_\rho f(x_0,\sigma),\ \ \ g \in G,\ x_0 \in F,\ \sigma \in \widehat{G},
\end{equation}
for $f \in L^2(X)$.

The inverse Zak transform $\calZ_\rho^{-1}$ is given by the extension of \eqref{Zak_inversion_nonabelian} to the image space.
\end{theorem}

\begin{proof}
The proof is essentially identical to that of Theorem \ref{thm:Zak_abelian}, using the non-abelian versions of the used results.

By Corollary \ref{cor:orbital_mean_L2}, the functions $f_{x_0}$ are in $L^2(G)$ for almost all $x_0 \in F$, so the Zak transform is well-defined almost everywhere by the $L^2$-version of eq. \eqref{Zak_Fourier_nonabelian}.

The Zak transform is an isometry by combining the Weil formula, the Plancherel theorem, eq. \eqref{integration_fundamental_domain} for the fundamental domain $F$, invariance of $x\mapsto \int_{\widehat{G}}\|\calZ_\rho^X(x,\sigma)\|_{\rm HS}\d\mu_{\widehat{G}}(\sigma)$, and the refined vanishing property. For $f \in L^1(X)\cap L^2(X)$, we get
\begin{align*}
\|f\|_2^2 &\stackrel{\rm Weil}{=} \int_{\rho\backslash X}\int_G|f_x|^2\d\mu_G\d\mu_{\rho\backslash X}(\pi_\rho(x)) \stackrel{\rm Plancherel}{=} \int_{\rho\backslash X}\int_{\widehat{G}}\|\widehat{f_x}\|_{\rm HS}^2\d\mu_{\widehat{G}}\d\mu_{\rho\backslash X}(\pi_\rho(x)) \\
&\stackrel{\eqref{Zak_Fourier_nonabelian}}{=} \int_{\rho\backslash X}\int_{\widehat{G}}\|\calZ_\rho f(x,\sigma)\|_{\rm HS}^2\d\mu_{\widehat{G}}(\sigma)\d\mu_{\rho\backslash X}(\pi_\rho(x)) \\
&\stackrel{\eqref{integration_fundamental_domain}}{=} \int_F\int_{\widehat{G}}\|\calZ_\rho^X(x,\sigma)\|_{\rm HS}^2\d\mu_{\widehat{G}}(\sigma)\d\mu_F(x) = \|\calZ_\rho f\|_{L^2_\rho(F\times\widehat{G})}^2.
\end{align*}
Thus, $\calZ_\rho$ can be extended to an isometry from $L^2(X)$ into $L^{2,\rm op}_\rho(F\times\widehat{G})$.

The proof of surjectivity works exactly as in Theorem \ref{thm:Zak_abelian}, using Corollary \ref{cor:Plancherel_quotient} for the left $G_{x_0}$-invariant functions $f_{x_0}$.

The intertwining property \eqref{Zak_intertwining_nonabelian} is again shown by direct computation, using the definition \eqref{Zak_nonabelian} and right invariance of $\mu_G$, which is due to unimodularity of $G$.
\end{proof}

There is a Hausdorff-Young inequality for second countable unimodular type I groups (see \cite{Kunze1958}) to prove an analog result for a Zak transform $\calZ_\rho:L^p(X)\to\calZ_\rho(L^p(X))\subseteq L^q(F\times\widehat{G})$ for $p \in (1,2)$ and $q = p/(p-1)$.

As in the abelian case (see \eqref{Zak_measure}), the reconstruction formula can be interpreted as a decomposition into measures. These operator-valued Zak measures are
\begin{equation} \label{Zak_measures_nonabelian}
\delta_G^{(x_0,\chi)}(\varphi) := \int_G\rho_g\varphi(x_0)\sigma(g)\d\mu_G(g) = \calZ_\rho\varphi(x_0,\sigma^*),\ \ \ \varphi \in C_c(X).
\end{equation}
They have the analog properties as their abelian counterparts and satisfy a weak inversion formula.

An important special case of Theorem \ref{thm:Zak_nonabelian} is the case that almost all stabilizers are trivial. In this case, the image of the Zak transform is the full space $\int_F^\oplus\int_{\widehat{G}}^\oplus\calH_\sigma\otimes\calH_{\overline{\sigma}}\d\mu_{\widehat{G}}(\sigma)\d\mu_F$.

\subsection{The character Zak transform and other alternatives}

The above generalization of the Zak transform is not the only possibility. As already mentioned in the Introduction, one can easily generalize to general actions of groups on $L^2(X)$. A few further approaches to non-abelian actions will be discussed in this section, since they might be of interest or have already been considered elsewhere.

One alternative to the Zak transform given in Definition \ref{def:Zak_nonabelian} uses the characters of the representations. This approach was already considered by Wigner \cite{Wigner1931}.

Let $G$ be an lcH group and $\sigma \in \widehat{G}$ a unitary irreducible representation. The character $\chi_\sigma:G\to\C$ of $\sigma$ is defined as
\begin{equation} \label{character}
\chi_\sigma(g) := \tr(\sigma(g)),\ \ \ g \in G.
\end{equation}

In the case that $G$ is abelian, this agrees with the standard definition in eq. \eqref{dual_group}, as the irreducible representations are one-dimensional. When $G$ is non-abelian, the characters are not group homomorphisms and transform in a more complicated way, namely
\begin{equation} \label{character_law_nonabelian}
\chi_\sigma(g^{-1}h) = \tr(\sigma(g)^*\sigma(h)) = \langle\sigma(h),\sigma(g)\rangle_{\HS},\ \ \ \sigma \in \widehat{G},\ g,h \in G.
\end{equation}
This is the price you pay for decomposing functions on a non-abelian structure into scalar functions.

We define the character Zak transform as follows.

\begin{definition}[Character Zak transform] \label{def:Zak_character}
Let $G$ be a second countable unimodular type I group that acts strongly proper on a Weil $G$-space $(X,\mu_{\rho\backslash X}^\#)$ via $\rho:G\times X\to X$, and $F$ a measurable fundamental domain of the action $\rho$. The character Zak transform $\calZ_\rho^\tr f$ of a function $f \in L^1(X,\mu_{\rho\backslash X}^\#)$ with $f_x \in L^1(G)$ for almost every $\pi_\rho(x) \in \rho\backslash X$ is defined as
\begin{equation} \label{Zak_character}
\calZ_\rho^\tr f(x_0,\sigma) := \int_G\rho_gf(x_0)\overline{\chi_\sigma(g)}\d\mu_G(g),\ \ \ x_0 \in F,\ \sigma \in \widehat{G}.
\end{equation}
Again the extended Zak transform $\calZ_\rho^{\tr,X}$ on $X\times\widehat{G}$ is given by eq. \eqref{Zak_character} with $x_0 \in X$.
\end{definition}

The character Zak transform is related to the Zak transform defined in Definition \ref{def:Zak_nonabelian} by simply taking the trace of the latter:
\begin{equation} \label{Zak_trace}
\calZ_\rho^\tr f(x_0,\sigma) = \tr(\calZ_\rho f(x_0,\sigma)),\ \ \ x_0 \in F,\ \sigma \in \widehat{G},
\end{equation}
where $f$ is as in Definition \ref{def:Zak_character}. From eqs. \eqref{character_law_nonabelian} and \eqref{Zak_trace}, we can infer the following transformation law of the extended character Zak transform
\begin{equation} \label{Zak_transformation_character}
\calZ_\rho^{\tr,X} f(\rho_g^{-1}(x_0),\sigma) = \langle\sigma(g),\calZ_\rho(x_0,\sigma)^*\rangle_\HS,\ \ \ g \in G,\ x_0 \in F,\ \sigma \in \widehat{G},
\end{equation}
for $f$ as above. A somewhat subtle consequence is that the function $f$ can only be reconstructed from the extended character Zak transform $\calZ_\rho^{\tr,X}f$ and not from $\calZ_\rho^\tr f$, namely
\begin{equation} \label{Zak_character_inverse}
f(x) = \int_{\widehat{G}}\calZ_\rho^{\tr,X}f(x,\sigma)\d\mu_{\widehat{G}}(\sigma),\ \ \ x \in X,
\end{equation}
which can be seen from eqs. \eqref{Zak_transformation_character} and \eqref{Zak_inversion_nonabelian}. Some information that is lost by taking the trace of $\calZ_\rho f$ can thus be retrieved from the transformation law of $\calZ_\rho^{\tr,X}f$. More precisely, the characters $\chi_\sigma$, $\sigma \in \widehat{G}$ are conjugation invariant, and only span the space of conjugation invariant $L^2$-functions (which are called class functions)\footnote{The resulting Fourier transform of class functions can be understood in more detail. In the case of a compact group, the sets of conjugacy classes and of characters, both have the structure of a hypergroup. These hypergroups are dual to each other (see \cite{Bloom1995},\cite{Lasser}).}.

Definition \ref{def:Zak_nonabelian} is not the only way of generalizing the Zak transform for abelian groups and actions.

The Zak transform on lca groups $G$ with a closed subgroup $H$ acting by right translation is usually defined in terms of $\widehat{G}/H_G^\bot$ instead of $\widehat{H}$. As already noted, the two objects are isomorphic (via the restriction operator) -- a fact that is essential for proving the Poisson summation formula for lca groups. The relevance of the Zak transform for sampling theory and Gabor analysis is largely due to its close relation to the Poisson summation formula (for details, see e.g. \cite{Groechenig2001}). In this spirit, one could try to define a Zak transform on non-abelian groups based on the restriction of the representations of $X = G$ to a closed subgroup $H$ that acts on $G$. These restricted representations are in general not irreducible, s.t. the resulting transform is not equivalent to our definition. We will not study this version of a Zak transform in more detail here, mainly due to the fact that one looses the relation to the Fourier transform on $H$. To study this transform, one would need to develop a harmonic analysis on quite general quotients $G/H$, similar to that developped in Section \ref{subsec:Poisson_compact}.

In \cite{Kutyniok2002}, Kutyniok bypasses the shortcomings of the dual $\widehat{G}$ of a non-abelian group $G$ by constructing an injective map $G\to \widehat{L}$, where $L$ is a suitable lca group that plays the role of $\widehat{G}$. In addition, a further group $Z$ and an action $\tau:G\to\Aut(L\times Z)$ are constructed, s.t. the group $G\ltimes_\tau(L\times Z)$ can be used as substitute for the Heisenberg group. Under certain technical conditions, this setting yields a Zak transform with most features of the abelian transform.

Finally, it should be noted that versions of Plancherel's theorem exist for more general situations than second countable unimodular type I groups, possibly allowing to generalize the versions in this article. For example, there is a rich Fourier theory for symmetric spaces (see, e.g., the survey article \cite{Ban1996}).

\section{Applications of the Zak transform} \label{sec:Zak_applications}

\subsection{A Bloch-Floquet theorem for group actions}
\label{subsec:Bloch_Floquet}

As noted in the Introduction, the classic Zak transform was considered in \cite{Zak1967} as a refinement of the decomposition of electron states in crystals into Bloch waves \cite{Bloch1929}, which are periodic functions that are suitably modulated. The Bloch decomposition relies on the fact that the partial differential equation under consideration -- the non-relativistic Schr\"odinger equation for a single electron -- is invariant under translation by a crystal lattice vector. In representation theoretic terms, the classic Zak transform diagonalizes the Hamiltonian. Even earlier, Floquet \cite{Floquet1883} applied the same argument to decompose ordinary differential equations that are invariant w.r.t. discrete translation groups in a similar manner. The recent survey article \cite{Kuchement2016} gives an overview of the theory and applications. The Zak transforms \eqref{Zak_abelian} and \eqref{Zak_nonabelian} provide the possibility to generalize the theory to operators that are invariant w.r.t. more general group actions.

\begin{corollary}[Bloch-Floquet theorem for strongly proper actions] \label{cor:Bloch_Floquet}
Let $(X,\mu_{\rho\backslash X}^\#)$ be a Weil $G$-space, and $F$ a measurable fundamental domain of the action $\rho$. Moreover, let $B(X)$ be a linear space of functions on $X$ that is $\rho$-invariant (e.g. a solution space of a differential equation on $X$ or an $L^p$-space), and on which the Zak transform is well-defined. Then every function $f \in B(X)$ can be decomposed as follows:
\begin{enumerate}
\item[(i)] When $G$ is an lca group, there are $\rho$-invariant functions $\calB_F^\chi:X\to\C$, $\chi \in \widehat{G}$, s.t.
\begin{equation} \label{Bloch_abelian}
f(x) = \int_{\widehat{G}}\calB_F^\chi(x)\chi(g)\d\mu_{\widehat{G}}(\chi),\ \ \ x \in X,\ \rho_g(x) \in F.
\end{equation}
The functions $x\mapsto\calB_F^\chi(x)\chi(g)$ are called Bloch waves and are given by
\begin{equation} \label{Bloch_Zak_abelian}
\calB_F^\chi(x) = \calZ_\rho f(x_0,\chi),\ \ \ x \in \rho_G(x_0).
\end{equation}
\item[(ii)] When $G$ is a non-abelian second countable unimodular type I group, there are $\rho$-invariant tensor fields $\calB_F^\sigma:X\to\calH_\sigma\otimes\calH_{\overline{\sigma}}$, $\sigma \in \widehat{G}$, on $X$, s.t.
\begin{equation} \label{Bloch_nonabelian}
f(x) = \int_{\widehat{G}}\tr(\calB_F^\sigma(x)\sigma(g))\d\mu_{\widehat{G}}(\sigma),\ \ \ x \in X,\ \rho_g(x) \in F.
\end{equation}
The tensor fields $x\mapsto\calB_F^\sigma(x)\sigma(g)$ are called Bloch tensor fields and are given by
\begin{equation} \label{Bloch_Zak_nonabelian}
\calB_F^\sigma(x) = \calZ_\rho f(x_0,\sigma),\ \ \ x \in \rho_G(x_0).
\end{equation}
\end{enumerate}
In both cases, when the space $B(X)$ is the solution space of an invariant differential equation, then the Zak transform intertwines the differential operator with an operator that acts component-wise on the invariant subspaces.
\end{corollary}

As an application of \ref{cor:Bloch_Floquet}, consider the proper action of a discrete subgroup $S$ of the Euclidean group $\E(3) = \R^3\rtimes\O(3)$ of isometries of $\R^3$. Molecular structures that are invariant w.r.t. such a group are called objective structures \cite{James2006} and are natural generalizations of crystals. Now, consider a single electron in an objective structure with symmetry group $S$. The non-relativistic Schr\"odinger equation for the state $\psi$ of the electron is of the form
\begin{equation} \label{Schroedinger}
H\psi := \left(-\Delta + V\right)\psi = E\psi,
\end{equation}
where $E$ is the energy, $\Delta$ is the Laplacian, and $V$ is the $\rho$-invariant potential. Let $\Psi^E$ be the space of $L^2$-solutions of eq. \eqref{Schroedinger}. To apply Corollary \ref{cor:Bloch_Floquet} to the space $\Psi^E$, we need to show that the group $S$ is type I (as a discrete group it is second countable and unimodular) and that the action is strongly proper. The latter follows from the result in \cite{Chabert2001}, since $S$ and $\R^3$ are both second countable. That $S$ is type I can be seen as follows.

Let $\rmT(S) := \{(\bfI|c) \in S \,|\, c \in \R^3\}$ be the subgroup of pure translations in $S$, where we used the notation $(\bfQ|c)$ for an element of $\E(3)$, where $\bfQ \in \O(3)$ and $c \in \R^3$, and the action on $\R^3$ is given by $\rho_{(\bfQ|c)}(x) = \bfQ x + c$, $x \in \R^3$. The group $\rmT(S)$ is always an abelian normal subgroup of $S$, since for $(\bfI|c) \in \rmT(S)$ and $(\bfQ|c') \in S$,
$$(\bfQ|c')(\bfI|c)(\bfQ|c')^{-1} = (\bfI|\bfQ c) \in \rmT(S).$$
Now, when the quotient $S/\rmT(S)$ is finite, then $S$ is type I due to \cite{Thoma1964}, since $\rmT(S)$ is an abelian normal subgroup of finite index. If $S/\rmT(S)$ is not finite, it needs to contain a screw displacement by an irrational angle, because these are the only elements of $\E(3)$ that generate infinite groups that don't contain translations. The only groups of the considered type that contain such an element are helical groups without translations (as shown in \cite{Dayal}). Moreover, the subgroup of all screw displacements of a helical group is an abelian normal subgroup of finite index, completing the proof. In summary we have shown:

\begin{proposition}[Discrete subgroups of $\E(3)$]
All discrete subgroups of $\E(3)$ are type I.
\end{proposition}

As a corollary, we get the following.

\begin{corollary}[Bloch/Zak decomposition of electron states in objective structures]
Every solution $\psi \in \Psi^E$ of eq. \eqref{Schroedinger} can be decomposed into Bloch tensor fields:
\begin{equation} \label{Bloch_OS}
\psi(x) = \int_{\widehat{S}}\tr(\calB_F^\sigma(x)\sigma(g))\d\mu_{\widehat{S}}(\sigma),\ \ \ \rho_g(x) \in F.
\end{equation}
where $\calB_F^\sigma$, $\sigma \in \widehat{S}$ is given by
\begin{equation} \label{Bloch_Zak_OS}
\calB_F^\sigma(x) = \calZ_\rho\psi(x_0,\sigma),\ \ \ x \in \rho_G(x_0).
\end{equation}
The Zak transform intertwines the Hamiltionian $H$ with a component-wise operator on the invariant subspaces.
\end{corollary}
Note that the existence of a nice fundamental domain $F$ is guaranteed by \cite{Bourbaki1989}, since $S$ is discrete and $\R^3$ is $\sigma$-compact.

The classic Zak transform is intimately related to energy bands in crystals. A rigorous derivation of this connection can be found in \cite{Reed1978}. The reasoning is as follows. The Zak transform intertwines the Schr\"odinger operator with a direct integral of operators that act on the irreducible subspaces of the image of the Zak transform. Now, the single operators have a discrete spectrum that varies continuously w.r.t. the wave vector in the Brillouin zone (character of the lattice subgroup). As a consequence, the union of all the spectra is the discrete union of intervals, which are called energy bands.

The interesting question, how this theory can be generalized to objective structures lies beyond the scope of this article and will be treated elsewhere.

The analog of the Bloch decomposition for certain discrete abelian and finite symmetries has been studied by Banerjee \cite{Banerjee2011,Banerjee2013}. He used this decomposition to design algorithms for density functional theory for objective structures.

\subsection{The Zak transform in radiation design}
\label{subsec:radiation_design}

The success of classic X-ray crystallography for the analysis of molecular structures is due to the highly structured and sharply peaked diffraction patterns of crystals. The mathematical reason underlying this phenomenon is the Poisson summation formula, and the patterns contain information on the symmetry of the crystal and -- up to a phase problem -- on the atomic structure of a fundamental domain.

The theory of radiation design asks the question, whether there is radiation that produces highly structured and sharply peaked diffraction patterns for structures that are not crystals. The answer is affirmative for helical structures like carbon nanotubes or filamentous viruses, where the `right' radiation are so-called twisted X-rays (see \cite{Friesecke2016,Juestel2016} for details).

For general objective structures, parts of the theory (namely for abelian and compact groups) have been developped in \cite{Juestel2014}, where the Zak transform emerges as a tool to interpret the diffraction patterns. The detailed theory for arbitrary objective structures will be derived elsewhere.

As a direct application of the Zak transform, we can quickly sketch the line of thought. According to the diffraction model in \cite{Friesecke2016}, the spatial part of the outgoing field when a sample with electron density $\varphi$ is illuminated with time-harmonic radiation is essentially (up to a decay factor and constants) given by the so-called radiation transform of $\varphi$
\begin{equation} \label{radiation_transform}
\calR[\bfE_0]\rho(s_0) = \int_{\R^3}\rmP(s_0^\bot)\bfE_0(x)e^{-i\frac{\omega}{c}s_0\cdot x}\varphi(x)\d x,
\end{equation}
where $\bfE_0$ is a complex vector field (the spatial part of the incoming electric field of frequency $\omega > 0$), $c$ is speed of light, $s_0 \in S^2$ is the outgoing direction, and $\rmP(s_0^\bot) := (\bfI - s_0s_0^T) \in \R^{3\times 3}$ is an orthogonal projection. The model relies on the assumptions of a weak incoming field, a high frequency, and observation in the far-field.

In classic X-ray crystallography, the incoming field is a plane wave $\bfE_0(x) = \bfE^{(k)}(x) :=  n e^{ik\cdot x}$, where $k \in \frac{\omega}{c}S^2$, and $n \in \C^3$, $n\cdot k = 0$. The radiation transform integral essentially reduces to a Fourier transform of $\varphi$ in this case. Now, assume that $\varphi$ is invariant w.r.t. a discrete isometry group $S\leq\E(3)$ to model the electron density of an objective structure. Moreover, assume that given $\sigma \in \widehat{S}$, we can define the symmetry projection $P^\sigma_S\bfE^{(k)}$ of a plane wave $\bfE^{(k)}$ (possibly in the distributional sense), where the Euclidean group acts on vector fields as follows. Let $g = (\bfQ|c)$ for $\bfQ \in \O(3)$ and $c \in \R^3$, and $\bfE$ a complex vector field on $\R^3$, then
\begin{equation} \label{action_vf}
\rho_g\bfE(x) := \bfQ\bfE(\bfQ^T(x-c)),\ \ \ x \in \R^3.
\end{equation}
The resulting object $P_S^\sigma\bfE^{(k)}$ can be considered as a field of elements of the tensor product $\calH_\sigma\otimes\calH_{\overline{\sigma}}\otimes\R^3$. By a formal calculation we can show that the radiation transform reduces to an integral over a fundamental domain $F$ of the action of $S$ on $\R^3$. For this purpose, we define the Zak transform $\calZ_\rho\bfE^{(k)}(x_0,\sigma) := P_S^\sigma\bfE^{(k)}(x)\sigma(g)$, $\rho_g(x) = x_0 \in F$. Then, writing $e^{(\ell)}(x) := e^{i\ell\cdot x}$ for $\ell \in \R^3$, we get the following formal calculation
\begin{align*}
\calR[P_S^{\sigma}\bfE^{(k)}]\varphi(s_0) &= \rmP(s_0^\bot)\int_{\R^3}P_S^\sigma\bfE^{(k)}(x)e^{-i\frac{\omega}{c}s_0\cdot x}\varphi(x)\d x \\
&= \rmP(s_0^\bot)\int_F\int_S \calZ_\rho^{\R^3}\bfE^{(k)}(\rho_g^{-1}(x_0),\sigma)e^{-i\frac{\omega}{c}s_0\cdot \rho_g^{-1}(x_0)}\rho_g\varphi(x_0)\d\mu_S(g)\d\mu_F(x_0) \\
&= \rmP(s_0^\bot)\int_F\int_S \calZ_\rho\bfE^{(k)}(x_0,\sigma)\sigma(g)\rho_g(e^{(-\frac{\omega}{c}s_0)}\cdot\varphi)(x_0)\d\mu_S(g)\d\mu_F(x_0) \\
&= \rmP(s_0^\bot)\int_F\calZ_\rho\bfE^{(k)}(x_0,\sigma)\overline{\calZ_\rho(e^{(\frac{\omega}{c}s_0)}\cdot\varphi)(x_0,\sigma)}\d\mu_F(x_0)
\end{align*}
Now, using a generalization of the isometry property of the Zak transform, integration over $\widehat{S}$ yields
\begin{align*}
\int_{\widehat{S}}\calR[P_S^{\sigma}\bfE^{(k)}]\varphi(s_0)\d\mu_{\widehat{S}}(\sigma) &= \rmP(s_0^\bot)\langle\calZ_\rho\bfE^{(k)},\calZ_\rho(e^{(\frac{\omega}{c}s_0)}\cdot\varphi)\rangle = \rmP(s_0^\bot)\langle \bfE^{(k)},e^{(\frac{\omega}{c}s_0)}\cdot\varphi\rangle \\
&= \rmP(s_0^\bot)n\widehat{\varphi}(\tfrac{\omega}{c}s_0-k),
\end{align*}
where the scalar products are defined component-wise. This is again essentially a Fourier coefficient of $\varphi$. In particular, if one could measure the radiation transform w.r.t. this symmetry-adapted radiation for every $\sigma \in \widehat{S}$, one could reconstruct the density $\varphi$. However, in diffraction experiments one can only measure the intensity of the outgoing radiation that is essentially given by the norm $\|\calR[P_S^{\sigma}\bfE^{(k)}]\varphi(s_0)\|_2^2$ of the radiation transform. This is similar to the phase problem, but it's a combined phase-and-orientation problem for a complex vector.

\appendix

\section{Proofs}
Two more or less standard proofs can be found in this appendix for the convenience of the reader.

\begin{proof}{\bf (of Lemma \ref{lem:orbital_mean_operator})}
For $x \in X$ and $f \in C_c(X)$, consider the function $f_x:\ G\to\C,\ g\mapsto \rho_gf(x)$.
The support of $f_x$ can be written as a transporter of two compact sets:
$$\supp(f_x) = G_{\supp(f),\{x\}} = \{g \in G \,|\, \rho_g(\supp(f))\cap\{x\}\not= \emptyset\}.$$
Consequently, by properness of the action, the function $f_x$ is compactly supported. It is also continuous by the following identity of set functions: $f_x^{-1} = \pi_1\circ\rho^{-1}\circ f^{-1}$, using continuity of $f$ and $\rho$, and the fact that the projection $\pi_1:G\times X\to G$, $(g,x)\mapsto g$, is open. So, in particular, $f_x \in C_c(G)\subset L^1(G)$, showing that the integral in the definition of $A_\rho f$ converges for every $x \in X$.

The function $A_\rho f$ is well-defined on $\rho\backslash X$ since the integral is constant on orbits. For $f \in C_c(X)$ and $h \in G$, by left-invariance of $\mu_G$,
\begin{align*}
\int_G \rho_gf(\rho_h^{-1}(x))\d\mu_G(g) &= \int_G f(\rho_{(hg)^{-1}}(x))\d\mu_G(g) = \int_G \rho_gf(x)\d\mu_G(g).
\end{align*}
It remains to show that $A_\rho f \in C_c(\rho\backslash X)$. The support $\supp(A_\rho f)$ is compact as a closed subset of the compact set $\pi_\rho(\supp(f))\subset \rho\backslash X$.

Continuity of $A_\rho f$ can be seen as follows. Consider a net $(x_k)_{k \in I}$ that converges to $x_0 \in X$. We show that $(A_\rho f(x_k))_{k \in I}$ converges to $A_\rho f(x_0)$. By convergence of $(x_k)_{k \in I}$ and local compactness of $X$, there is a compact set $K_0\subset X$ that contains $x_0$, and s.t. $x_k \in K_0$ for all $k\geq k_0$ for some $k_0 \in I$. As seen above, for $x \in K_0$, we have
$$\supp(f_x) = G_{\supp(f),\{x\}} \subseteq \bigcup_{x \in K_0}G_{\supp(f),\{x\}} = G_{\supp(f),K_0},$$
which is a compact set by properness of the action. Also, the function
$$G_{\supp(f),K_0}\times \supp(f)\to \C,\ (g,x)\mapsto \rho_gf(x),$$
is continuous and compactly supported, and thus uniformly continuous by the Heine-Cantor theorem. In particular, the net $(f_{x_k})_{k \in I}$ converges to $f_{x_0}$ uniformly on $G_{\supp(f),K_0}$. Consequently, we can take the limit out of the integral to obtain
$$A_\rho f(x_0) = \int_G f_{x_0}\d\mu_G = \int_{G_{\supp(f),K_0}}\lim_{k\in I}f_{x_k}\d\mu_G = \lim_{k\in I}\int_{G_{\supp(f),K_0}}f_{x_k}\d\mu_G = \lim_{k\in I}A_\rho f(x_k)$$
for all $k \geq k_0$, showing continuity of $A_\rho f$.
\end{proof}

\begin{proof}{\bf (of Theorem \ref{thm:Weil_formula})} The proof generalizes the arguments of the proofs of the special case of homogeneous spaces $G/H$ given in \cite{Folland1995,Reiter2000}.

We start by investigating  the functions $\lambda_g$, $g \in G$, that are associated to the $\rho$-quasi-invariance of $\mu_X$. Let $f \in C_c(X)$ and $g_1,g_2 \in G$, then
\begin{align*}
\int_X f\cdot\lambda_{g_1g_2}\d\mu_X &= \int_X \rho_{g_1g_2}f\d\mu_X = \int_X \rho_{g_2}f\cdot \lambda_{g_1}\d\mu_X = \int_X f\cdot\rho_{g_2^{-1}}\lambda_{g_1}\cdot\lambda_{g_2}\d\mu_X.
\end{align*}
Since $f$ was arbitrary, this shows that
\begin{equation} \label{star}
\rho_{g_2}\lambda_{g_1} = \frac{\lambda_{g_1g_2^{-1}}}{\lambda_{g_2^{-1}}}\ \ \ \mbox{for}\ g_1,g_2 \in G. \tag{$*$}
\end{equation}
Now, define the function $q:X\to (0,\infty)$ by
\begin{equation} \label{q}
q(x) := \int_G\rho_g\beta(x)\frac{\lambda_{g^{-1}}(x)}{\Delta_G(g)}\d\mu_G(g),\ \ \ x \in X,
\end{equation}
where $\beta$ is a Bruhat function for $\rho$, which exists by Lemma \ref{lem:Bruhat_function}. The function $q$ satisfies the functional equation \ref{functional_eq_q}, since for $h \in G$ and $x \in X$,
\begin{align*}
\rho_hq(x) &= \int_G\rho_{hg}\beta(x)\frac{\rho_h\lambda_{g^{-1}}(x)}{\Delta_G(g)}\d\mu_G(g) = \int_G\rho_{g}\beta(x)\frac{\rho_h\lambda_{g^{-1}h}(x)}{\Delta_G(h^{-1}g)}\d\mu_G(g) \\
&\stackrel{(*)}{=} \int_G\rho_{g}\beta(x)\frac{\lambda_{g^{-1}hh^{-1}}(x)}{\Delta_G(h^{-1})\Delta_G(g)\lambda_{h^{-1}}(x)}\d\mu_G(g) = \frac{\Delta_G(h)}{\lambda_{h^{-1}}(x)}q(x).
\end{align*}
Now, we define a positive linear functional $I_{\rho\backslash X}$ on $C_c(\rho\backslash X)$ that will yield the desired measure $\mu_{\rho\backslash X}$ via the Riesz representation theorem. For $f \in C_c(\rho\backslash X)$, we define
$$I_{\rho\backslash X}(f) := \int_X\widetilde{f}\cdot q\d\mu_X,\ \ \ \mbox{with}\ \widetilde{f} \in C_c(X),\ \mbox{s.t.}\ A_\rho\widetilde{f} = f.$$
Note that such a function $\widetilde{f}$ exists by surjectivity of $A_\rho$.

We need to show that this functional is independent of the choice of $\widetilde{f}$. First, given a Bruhat function $\beta$, we can choose the preimage $\beta\cdot\widetilde{f}_\rho$ instead of $\widetilde{f}$, where $\widetilde{f}_\rho(x) := A_\rho\widetilde{f}(\pi_\rho(x)) = f(\pi_\rho(x))$. Using Fubini and the functional equation for $q$, we get
\begin{align*}
\int_X&(\beta\cdot\widetilde{f}_\rho)\cdot q\d\mu_X = \int_X\int_G\beta\cdot\rho_g\widetilde{f}\cdot q\d\mu_G(g)\d\mu_X = \int_X \widetilde{f}\cdot\int_G\rho_{g^{-1}}\beta\cdot\rho_{g^{-1}}q\cdot\lambda_g\d\mu_G(g) \d\mu_X \\
&\stackrel{\eqref{functional_eq_q}}{=} \int_X\widetilde{f}\cdot\int_G \rho_{g^{-1}}\beta\cdot\frac{\Delta_G(g^{-1})}{\lambda_g}q\cdot\lambda_g\d\mu_G(g) \d\mu_X = \int_X \widetilde{f}\cdot q\cdot\underbrace{\int_G\rho_g\beta\d\mu_G(g)}_{= \varphi_\rho = 1}\d\mu_X = I_{\rho\backslash X}(f).
\end{align*}

Now, for two functions $\widetilde{f}_1,\widetilde{f}_2 \in C_c(X)$ with $A_\rho\widetilde{f}_1 = A_\rho\widetilde{f}_2 = f$, we infer
\begin{align*}
\int_X&\widetilde{f}_1\cdot q\d\mu_X - \int_X\widetilde{f}_2\cdot q\d\mu_X  = \int_X(\widetilde{f}_1-\widetilde{f}_2)\cdot q\d\mu_X = \int_X(\widetilde{f}_1 - \widetilde{f}_2)\cdot q\d\mu_X \\
&= \int_X(\beta\cdot(\widetilde{f}_1)_\rho - \beta\cdot(\widetilde{f}_2)_\rho)\cdot q\d\mu_X = \int_X(\beta(x)\cdot A_\rho\widetilde{f}_1(\pi_\rho(x)) - A_\rho\widetilde{f}_2(\pi_\rho(x))\cdot q(x)\d\mu_X(x) \\
&= \int_X\beta(x)(f(\pi_\rho(x)) - f(\pi_\rho(x)))\cdot q(x)\d\mu_X(x) = 0,
\end{align*}
showing that $I_{\rho\backslash X}$ is well-defined.

By the Riesz representation theorem, there is a unique Radon measure $\mu_{\rho\backslash X}$ on $\rho\backslash X$, s.t. $I_{\rho\backslash X}(f) = \int_{\rho\backslash X}f\d\mu_{\rho\backslash X}$. In particular, for any $f \in C_c(X)$ we have
$$\int_X f\cdot q\d\mu_X = I_{\rho\backslash X}(A_\rho f) = \int_{\rho\backslash X}A_\rho f\d\mu_{\rho\backslash X},$$
i.e. the Weil formula \eqref{Weil_formula} holds.

When $\mu_X$ is relatively $\rho$-invariant with $\lambda_g\equiv \Delta_G(g^{-1})$, then $q\equiv 1$ satisfies the functional equation, yielding a measure s.t. \eqref{Weil_formula} holds.
\end{proof}

\vspace{1cm}
{\bf Acknowledgement.}
The author thanks Gero Friesecke, Rupert Lasser, and Arash Ghaani Farashahi for helpful discussions on various topics associated to this work, and the organizers of the HIM Trimester Program ``Mathematics of Signal Processing" for inviting me and for their hospitality.

\bibliographystyle{amsalpha}

\bibliography{zak}

\end{document}